\documentclass{amsart}

\usepackage{amsmath,amsthm,amssymb,amsfonts}
\usepackage{bbm}
\usepackage[mathscr]{eucal}
\usepackage[english]{babel}
\usepackage{graphicx}

\allowdisplaybreaks[1]

\newtheorem{theorem}{Theorem}[section]
\newtheorem{lemma}[theorem]{Lemma}

\newtheorem{proposition}[theorem]{Proposition}
\newtheorem{definition}[theorem]{Definition}

\newtheorem{remark}[theorem]{Remark}

\numberwithin{equation}{section}

\def\bB{{\mathbb B}}

\def\bN{{\mathbb N}}

\def\bR{{\mathbb R}}
\def\bS{{\mathbb S}}
\def\bT{{\mathbb T}}

\def\R{{\mathbb R}}

\def\cD{\mathcal{D}}
\def\cE{\mathcal{E}}

\def\cH{\mathcal{H}}

\def\cP{\mathcal{P}}
\def\cQ{\mathcal{Q}}

\def\cV{\mathcal{V}}

\def\supp{\operatorname{supp}}
\def\Exp{\operatorname{Exp}}
\def\diam{\operatorname{diam}}
\def\diver{\operatorname{div}}

\def\eps{\varepsilon}

\def\ONE{{\mathbbm 1}}
\def\tw{\tilde{w}}
\def\ww{\breve{w}}
\def\tcP{\tilde{\cP}}
\def\tcV{\tilde{\cV}}
\def\qq{q}

\begin{document}

\title[Gaussian bounds for heat kernels on the ball and simplex]
{Gaussian bounds for the weighted heat kernels \\on the interval, ball and simplex}

\author[G. Kerkyacharian]{Gerard Kerkyacharian}
\address{LPSM, CNRS-UMR 7599, and Crest}
\email{kerk@math.univ-paris-diderot.fr}

\author[P. Petrushev]{Pencho Petrushev}
\address{Department of Mathematics, University of South Carolina, Columbia, SC 29208}
\email{pencho@math.sc.edu}

\author[Y. Xu]{Yuan Xu}
\address{Department of Mathematics, University of Oregon, Eugene, Oregon 97403-1222}
\email{yuan@math.uoregon.edu}

\thanks{The first author has been supported by ANR Forewer.
The second author has been supported by NSF Grant DMS-1714369.
The third author has been supported by NSF Grant DMS-1510296.}
\thanks{Corresponding author: Pencho Petrushev, E-mail: pencho@math.sc.edu}

\begin{abstract}
The aim of this article is to establish two-sided Gaussian bounds for the heat kernels on the unit ball and simplex in $\R^n$,
and in particular on the interval,
generated by classical differential operators whose eigenfunctions are algebraic polynomials.
To this end we develop a general method that employs the natural relation of such operators with weighted
Laplace operators on suitable subsets of Riemannian manifolds
and the existing general results on heat kernels.
Our general scheme allows to consider heat kernels in the weighted cases on the interval, ball, and simplex with parameters in the full range.
\end{abstract}

\date{January 20, 2018}

\maketitle

\tableofcontents

\section{Introduction}

We establish two-sided Gaussian bounds for the heat kernels
generated by classical differential operators in weighted cases
on the unit ball and simplex in $\bR^n$ and, in particular on the interval,
whose eigenfunctions are algebraic polynomials.
One of our principle examples is the operator
\begin{equation}\label{def-L-ball-intro}
L:=\sum_{i=1}^{n}\partial_i^2 - \sum_{i=1}^n \sum_{j=1}^n x_ix_j  \partial_{i}\partial_j  - (n+2\gamma)\sum_{i=1}^{n} x_i\partial_i,
\quad \gamma >-1/2,
\end{equation}
on the unit ball $\bB^n\subset\bR^n$ %, $n\ge 2$,
equipped with the measure
$d\mu(x) := (1-\| x\|^2)^{\gamma-1/2} dx$
and the distance
\begin{equation*}%\label{dist-ball}
\rho(x,y) := \arccos \big(x\cdot y + \sqrt{1-\| x\|^2}\sqrt{1-\| y\|^2}\big),
\end{equation*}
where $x\cdot y$ is the inner product of $x, y\in \bR^n$
and $\|x\|$ is the Euclidean norm of $x$.
As will be seen the operator $L$ is symmetric and $-L$ is positive.

Denote by $\tilde{\cV}_k$ the set of all algebraic polynomials of degree $k$ that are orthogonal
in $L^2(\bB^n, \mu)$ to lower degree polynomials and let $\tilde{\cV}_0$ be the set of all constants.
As is well known (see e.g. \cite[\S2.3.2]{DX}) $\tilde{\cV}_k$, $k=0, 1, \dots$, are eigenspaces of the operator $L$,
namely,
\begin{equation*}%\label{eigen-sp}
L\tilde P=-\lambda_k \tilde P,\quad \forall \tilde P\in\tilde{\cV}_k, \;\;\hbox{where $\lambda_k:= k(k+n+2\gamma-1)$}.
\end{equation*}
Let $\tilde P_k(x, y)$ be the kernel of the orthogonal projector onto $\tilde{\cV}_k$.
Then the semigroup $e^{tL}$, $t>0$, generated by $L$ has a (heat) kernel $e^{tL}(x,y)$ of the form
\begin{equation*} %\label{ball-HK}
e^{tL}(x,y)=\sum_{k=0}^\infty e^{-\lambda_k t}\tilde P_k(x, y).
\end{equation*}
We establish two-sided Gaussian bounds on $e^{tL}(x,y)$ of the form:
\begin{equation} \label{gauss-intro}
\frac{c_1\exp\{- \frac{\rho(x,y)^2}{c_2t}\}}{\big[V(x, \sqrt t) V(y, \sqrt t)\big]^{1/2}}
\le  e^{tL}(x,y)
\le \frac{c_3\exp\{- \frac{\rho(x,y)^2}{c_4t}\}}{\big[V(x, \sqrt t) V(y, \sqrt t)\big]^{1/2}}.
\end{equation}
Here $V(x, r):=\mu(B(x, r))$ is the volume of the ball $B(x, r)$ centered at $x$ of radius~$r$.
It is important to point out that in the literature the parameter $\gamma$ in \eqref{def-L-ball-intro}
is invariably restricted to $\gamma\ge 0$. Our method allows to operate in the full range $\gamma>-1/2$.

We obtain a similar result on the simplex
$\bT^n:=\big\{x \in \bR^n: x_i > 0, |x|< 1\big\}$,
$|x|:= \sum_i x_i$,
with weight
$\prod_{i=1}^{n} x_{i}^{\kappa_i-1/2}(1-|x|)^{\kappa_{n+1}-1/2}$, $\kappa_i >-1/2$,
and as a consequence for the Jacobi heat kernel
on $[-1, 1]$ with weight $(1-x)^\alpha(1+x)^\beta$, $\alpha, \beta > -1$.

Note that two-sided Gaussian bounds for the Jacobi heat kernel are also established in \cite[Theorem~7.2]{CKP}.
In \cite{NS} Nowak and Sj\"{o}gren obtained this result in the case when $\alpha, \beta \ge -1/2$
via a direct method using special functions.

In \cite{KPX} we derived two-sided Gaussian bounds for the heat kernels
on the ball and simplex as in \eqref{gauss-intro} from the Jacobi case
under the restrictions $\gamma\ge 0$ for the ball and $\kappa_i\ge 0$ for the simplex.

To prove our results on the ball and simplex we first develop a general method that employs the natural relation
between differential operators on open relatively compact subsets of $\R^n$
whose eigenfunctions are algebraic polynomials and
weighted Laplace operators on respective subsets of Riemannian manifolds
and then utilize existing results on two-sided Gaussian bounds for heat kernels on manifolds.
Our development heavily relies on a general result of Gyrya and Saloff-Coste from \cite{GS}
on the heat kernel in Harnack-type Dirichlet spaces with Neumann boundary conditions in inner uniform domains.
We apply the result from \cite{GS} in the particular case of a bilinear Dirichlet form generated by weighted Laplacian on
an open relatively compact convex subset of a ``good" Riemannian manifold.
In the process we establish some basic properties of convex subsets of Riemannian manifods.
In particular, we show that any open relatively compact convex subset of a Riemannian manifold
is an inner uniform domain.
As a result we establish Gaussian bounds on the related heat kernels just as in \eqref{gauss-intro}

A crucial step in this undertaking is to show that the classical differential operators
of interest on the ball or simplex whose eigenfunctions are algebraic polynomials
are naturally related through charts to weighted Laplace operators
on appropriate subsets of the unit sphere in $\bR^{n+1}$, considered as a Riemannian manifold.
This intimate relation enables us to deploy our general result and show that
an operator $L$ like these is essentially self-adjoint and $-L$ is positive,
and more importantly that the associated semigroup $e^{tL}$ has a (heat) kernel with two-sided Gaussian bounds
as in \eqref{gauss-intro}.

It is an open problem to identify other particular settings where the utilization of our method can
produce Gaussian bounds for the respective heat kernels.

The two-sided Gaussian bounds on heat kernels have a great deal of applications
in Harmonic Analysis, PDEs, Probability, and elsewhere.
For example, as is shown in \cite{CKP, GKPP} they allow to develop the theory of Besov and Triebel-Lizorkin spaces
with complete range of indices in the setting of Dirichlet spaces with doubling measure and local Poincar\'{e} inequality.
The Gaussian heat kernel estimates from this article imply that the results from \cite{CKP, GKPP} generalize
the ones on the interval, ball, and simplex from \cite{PX1, PX2, KPX1, KPX2, IPX1}.
Furthermore, these results break new ground
%Our heat kernel Gaussian estimates on the interval, ball and simplex break new ground
in allowing to extend all results from \cite{PX1, PX2, KPX1, KPX2, IPX1} to the full range of the parameters
of the weights.

An interesting specific consequence of the upper Gaussian bound on heat kernels is the {\em finite speed propagation property},
which plays an important role e.g. in the development of smooth functional calculus in \cite{GKPP}.
This important property is not well known for the interval, ball or simplex.
We state it on the ball in \S\ref{sec:ball}.
This property is essentially used in \cite{KyrP} for the construction of frames on the ball
with small shrinking supports.

The organization of the paper is as follows.
In \S\ref{sec:general} we develop our general method for establishing two-sided Gaussian bounds
for heat kernels associated with differential operators that are realizations of weighted Laplace operators
on suitable charts of Riemannian manifolds.
This include the presentation of the need result by Gyrya and Saloff-Coste \cite{GS} in the specific case of Riemannian manifolds,
establishment of basic properties of convex subsets of Riemannian manifolds,
development of our setting,
and the proof of the main result.
In \S\ref{sec:ball} we apply our general result from \S\ref{sec:general} to obtain
two-sided Gaussian bounds for the weighted heat kernel on the unit ball in $\R^n$.
We also present some consequences of this result.
In \S\ref{sec:simplex} we obtain two-sided Gaussian bounds on the weighted heat kernel on the simplex in $\R^n$.
Finally, in \S\ref{sec:jacobi} we derive Gaussian bounds for the Jacobi heat kernel from the case of the simplex.

\smallskip

\noindent
{\em Notation:}
The following notation will be useful
$a\wedge b := \min\{a, b\}$,
$a\vee b := \max\{a, b\}$.
Positive constants will be denoted by $c, c', c_0, c_1, \dots$ and they may very at every occurrence;
$a\sim b$ will stand for $c_1\le a/b\le c_2$.
Most constants will depend on some parameters that will be clear from the context.

In this article all functions that we deal with are assumed to be real-valued.

\section{General result on heat kernels with Gaussian bounds}\label{sec:general}

In this section we develop our idea for establishing two-sided Gaussian bounds on heat kernels
generated by operators that are realizations of weighted Laplace operators
in local coordinates on suitable charts of Riemannian manifolds.

\subsection{Heat kernel on Riemannian manifolds and their open convex subsets}\label{subsec:heat-ker}

As was explained in the introduction it will be critical for our development that
the operator $L$ of interest is a realization of a weighted Laplace operator in local coordinates
on a suitable chart of a Riemanian manifold.
In this section we collect all facts that we need on Riemannian manifolds.
We refer the reader to \cite{Grig} for details.

\subsubsection{Heat kernel on Riemannian manifolds}\label{subsubsec:riem-manifold}

Assume that $M$ is a complete $n$-dimensional Riemannian manifold
and let $\nu$ be the Riemannian measure.
As usual the distance on $M$ will be the geodesic distance $d(\cdot, \cdot)$ on M.
We denote by $V(x, r)$ the volume of the ball of radius $r>0$ centered at $x\in M$, that is,
\begin{equation*}
V(x, r):= \nu(B(x, r)), \quad B(x, r):=\{y\in M: d(y, x)<r\}.
\end{equation*}

As usual we denote by $T_xM$ the tangent space of $M$ at $x$ and by $T_x^*M$ its dual.
Set $TM:=\cup_x T_xM$.
We denote by $g(x)(\cdot, \cdot)$ the Riemannian metric tensor.
This is a symmetric positive definite bilinear form on $T_xM$ that depends smoothly on $x\in M$.
Then
\begin{equation}\label{def-inner-prod}
\langle \xi, \eta\rangle_g := g(x)(\xi, \eta), \quad \xi, \eta\in T_xM,
\end{equation}
is an inner product on $T_xM$.
Denote $|\xi|_g:= \sqrt{\langle \xi, \xi\rangle_g}$.

Denote by $C(M)$ be the space of continuous functions on $M$ and
by $C_c(M)$ the space of all functions $f\in C(M)$ with compact support.
Also, denote
\begin{equation}\label{def:DM}
\cD(M):= C^\infty(M) \cap C_c(M).
\end{equation}
Further, we denote by
$\overrightarrow{C}^\infty(M)$ the space of smooth vector fields $\vec{v}\in TM$
and by
$\overrightarrow{\cD}(M)$ the space of all $\vec{v} \in \overrightarrow{C}^\infty(M)$ with compact support.

The gradient and divergence operators will be denoted by $\nabla$ and $\diver$.
As is well known
$\nabla :  C^\infty(M) \mapsto \overrightarrow{C}^\infty(M)$
and
$\diver:  \overrightarrow{C}^\infty(M) \mapsto   C^\infty(M)$.
The divergence theorem \cite[Theorem 3.14]{Grig} asserts that for any vector field $\vec{v}\in\overrightarrow{C}^\infty(M)$
there exists a unique function $\diver \vec{v}\in C^\infty(M)$ such that
\begin{equation}\label{divergence0}
\int_Mu\diver \vec{v} d\nu = -\int_M \langle \vec{v}, \nabla u\rangle_g d\nu,
\quad \forall u\in \cD(M).
\end{equation}
This identity also holds if $u\in C^\infty(M)$ and $\vec{v}\in \overrightarrow{\cD}(M)$
(see \cite[Corollary 3.15]{Grig}).

The Laplace (or Laplace-Beltrami) operator $\Delta$ on $M$ is defined by
\begin{equation*}
\Delta f:= \diver(\nabla f), \quad f\in C^\infty(M).
\end{equation*}
Identity \eqref{divergence0} yields the following Green's formula:
If $f, h\in C^\infty(M)$ and $f\in \cD(M)$ or $h\in \cD(M)$, then
\begin{equation}\label{diverge-2}
\int_M f\Delta h d\nu = -\int_M\langle\nabla f, \nabla h \rangle_g d\nu = \int_M h\Delta f d\nu.
\end{equation}

\noindent
{\bf Self-adjoint extensions of the Laplace operator.}
We next consider the Dirichlet and Neumann extensions of the Laplace operator $\Delta$ on $M$.

We first introduce the adjoint operator $\Delta^*$ of $\Delta$.
We consider the operator $\Delta$ defined on $\cD(M)$ that is dense in $L^2(M, \nu)$.
The domain $D(\Delta^*)$ of $\Delta^*$ is defined as the set of all $f \in L^2(M, \nu)$ for which there exists $h\in L^2(M, \nu)$
such that
$$
\int_M f \Delta \theta d\nu = \int_M h \theta d\nu, \quad \forall \theta \in \cD(M).
$$
For each $f\in D(\Delta^*)$ one defines $\Delta^* f :=h$.
By \eqref{diverge-2} it readily follow that $\Delta$ is symmetric and $-\Delta$ is positive.
Therefore, the adjoint operator $\Delta^*$ is closed
and $\Delta \subset \Delta^*$.

\smallskip

\noindent
{\bf Dirichlet Laplacian $\Delta^D$.}
We introduce the quadratic form
\begin{equation*}
\cE^D(f,h):= \int_M \langle \nabla f, \nabla h \rangle_g d\nu
\quad\hbox{with domain}\quad
D(\cE^D):= \cD(M)
\end{equation*}
and associated norm
\begin{equation*}
\|f\|_{\cE^D}^2:= \|f\|_{L^2}^2 + \cE^D(f,f).
\end{equation*}
It is not hard to see that $\cE^D$ is closable.
We denote by $\overline{\cE^D}$ the closure of $\cE^D$ and by $W^D=D(\overline{\cE^D})$ its domain.

Further, we define the domain of the Dirichlet Laplacian by
\begin{equation*}
D(\Delta^D):=\big\{f\in W^D: |\overline{\cE^D}(f, \theta)|\le c\|\theta\|_{L^2}, \; \forall \theta\in \cD(M)\big\}
\end{equation*}
and define $\Delta^D f$ for $f\in D(\Delta^D)$ from the identity
\begin{equation}\label{def:laplace-D}
\int_M (\Delta^Df)\theta d\mu = - \overline{\cE^D}(f, \theta),
\quad
\forall \theta\in\cD(M).
\end{equation}
In other words
\begin{equation*}%\label{def:laplace-ext-D}
D(\Delta^D):= D(\overline{\cE^D})\cap D(\Delta^*)
\quad\hbox{and}\quad
\Delta^Df:= \Delta^*f, \;\;\forall f \in D(\Delta^D).
\end{equation*}
The point is that $\Delta^D$ is a self-adjoint (Friedrichs) extension of $\Delta$.

\smallskip

\noindent
{\bf Neumann Laplacian $\Delta^N$.}
We now consider the quadratic form
\begin{equation*}
\cE^N(f,h):= \int_M \langle \nabla f, \nabla h \rangle_g d\nu
\end{equation*}
with domain
$ %\begin{equation*}
D(\cE^N):= \Big\{f\in L^2(M)\cap C^\infty(M): \int_M|\nabla f|^2_g d\nu<\infty \Big\}
$ %\end{equation*}
and associated norm
\begin{equation*}
\|f\|_\cE^2:= \|f\|_{L^2}^2 + \cE^N(f,f).
\end{equation*}
It is easy to see that $\cE^N$ is closable.
We denote by $\overline{\cE^N}$ the closure of $\cE^N$ and by $W^N=D(\overline{\cE^N})$ its domain.

Similarly as above, we define the domain of the Neumann Laplacian $\Delta^N$ by
\begin{equation*}
D(\Delta^N):=\big\{f\in W^N: |\overline{\cE^D}(f, \theta)|\le c\|\theta\|_{L^2}, \; \forall \theta\in \cD(M)\big\}
\end{equation*}
and define $\Delta^N$ from the identity
\begin{equation}\label{def:laplace-N}
\int_M (\Delta^Nf)\theta d\mu = - \overline{\cE^N}(f, \theta),
\quad
\forall D(\Delta^N), \; \forall \theta\in\cD(M).
\end{equation}
It is important that $\Delta^N$ is a self-adjoint extension of $\Delta$.
For more details, see \cite{FUKU}.

\smallskip

From our assumption that the Riemannian manifold $M$ is {\em complete}
it follows that
\begin{equation}\label{D-equivalent-D}
W^D=W^N \quad\hbox{and, therefore,}\quad \Delta^D= \Delta^N,
\end{equation}
see \cite{Grig}, Chapter 11.

\begin{remark}\label{rem:D-form}
Using the terminology from \cite{GS} we can claim that
$(\overline{\cE^N}, W^N)$ is a strictly local regular Dirichlet form.
Hence, the associated semi-group $e^{t\Delta^N}$, $t>0$, is a sub-Markovian strongly continuous semi-group.
\end{remark}

\noindent
{\bf Fundamental assumption.}
We will stipulate two key conditions on the Riemannian manifold $(M, d, \nu)$ we deal with:

(a) {\em The volume doubling condition}:
There exists a constant $c_0>0$ such that
\begin{equation}\label{doubling-0}
V(x, 2r)\le c_0V(x, r), \quad \forall x\in M, \;\forall r>0.
\end{equation}

(b) {\em Poincar\'{e} inequality}:
There exists a constant $P_0>0$ such that
\begin{equation}\label{poincare}
\int_{B(x, r)}|f-f_B|^2 d\nu \le P_0r^2\int_{B(x, r)}|\nabla f|_g^2d\nu,
\quad \forall f\in \cD(M), \; \forall x\in M, \forall r>0,
\end{equation}
where $f_B:=V(x, r)^{-1}\int_{B(x, r)} fd\nu$.

As is well known (see \cite{Grig1, {S-C1}} and also \cite{S-C2}) conditions (a)-(b) are equivalent to %the claim:
two-sided Gaussian bounds on the heat kernel:
$e^{t\Delta^N}$, $t>0$, is an integral operator with kernel $e^{t\Delta^N}(x, y)$
such that for any $x,y \in M$ and $t>0$
\begin{equation}\label{gauss-main-0}
\frac{c_1\exp\{- \frac{d(x,y)^2}{c_2t}\}}{\big[V(x, \sqrt t)  V(y, \sqrt t)\big]^{1/2}}
\le e^{t\Delta^N}(x, y)
\le \frac{c_3\exp\{- \frac{d(x,y)^2}{c_4t}\}}{\big[V(x, \sqrt t) V(y, \sqrt t)\big]^{1/2}}.
\end{equation}
Here  $c_1, c_2, c_3, c_4 >0$ are constants.

\subsubsection{Weighted Laplace operator in chart of Riemannian manifold}\label{subsec:weighted-laplace}

We adhere to the setting and notation introduced in the previous subsection.
In addition, we assume that $M\subset \bR^m$ and the Riemannian metric on $M$ is induced by the inner product on $\bR^m$.
It will be convenient to us to use the notation $y=(y_1, \dots, y_m)$ for points on $M\subset \R^m$
and $v=(v^1, \dots, v^n)$ for vectors in the tangent space $T_yM$.

Our goal is to show how two-sided Gaussian bounds can be obtained in the case of
a heat kernel generated by weighted Laplace operator $\Delta_w$ on an open relatively compact subset $U$ of $M$.

Assume that $(U, \varphi)$ is a chart on $M$, where $U$ is a connected open relatively compact subset of $M$ such that
$\varphi$ maps diffeomorphically $U$ onto $V$,
where $V\subset \bR^n$.

It will be convenient to work with the map $\phi:=\varphi^{-1}$.
Thus $\phi: V\to U$ is a $C^\infty$ bijection and in ``local coordinates"
\begin{equation*}
\phi(x)=(\phi_1(x), \dots, \phi_m(x))\in U\subset \bR^m, \quad\forall x\in V\subset \bR^n.
\end{equation*}
The Riemannian tensor $g(x)=(g_{ij}(x))$ can be represented by
\begin{equation}\label{R-tensor}
g(x)_{ij}= \big\langle \partial_i \phi(x) ,  \partial_j \phi(x)\big\rangle_{\bR^m}
= \sum_{k=1}^m \partial_i \phi_k(x)\partial_j \phi_k(x),
\quad x\in V, \; 1\le i, j\le n.
\end{equation}
As usual we shall denote by $g^{-1}(x)=(g^{ij}(x))$ the inverse of $g(x)$.

\smallskip

A particular case of a simple but useful map $\phi$ is considered in the following

\begin{proposition}\label{prop:simpl-map}
In the setting from above, assume that the map $\phi:V\to U$, $V\subset \bR^n$, $U\subset \bR^{n+1}$,
is of the form
\begin{equation*}
\phi(x)=(x_1, x_2, \dots, x_n, \psi(x)).
\end{equation*}
Then
$ 
g_{ij}(x)=\delta_{ij}+\partial_i\psi(x)\partial_j\psi(x),
$ 
\begin{equation}\label{inverse-g}
g^{ij}(x)=\delta_{ij}-\frac{\partial_i\psi(x)\partial_j\psi(x)}{1+\sum_\ell|\partial_\ell\psi(x)|^2},
\end{equation}
and
\begin{equation}\label{determ-g}
\det g(x)= 1+\sum_\ell|\partial_\ell\psi(x)|^2.
\end{equation}
\end{proposition}
\begin{proof}
Denote $F_i:=\partial_i\psi(x)$ and consider $F:=(F_1, \dots, F_n)^T$ as a vector in $\bR^n$.
Assume $F\ne 0$.
By \eqref{R-tensor} it readily follows that
$g_{ij}(x)=\delta_{ij}+\partial_i\psi(x)\partial_j\psi(x)$
and hence
$g(x)={\rm Id}+FF^T$.
Denote $P:= \|F\|^{-2}FF^T$.
Clearly, $P$ is the matrix of the orthogonal projector onto the one dimensional space spanned by $F$,
that is, $PF=F$ and $PV=0$ if $V\perp F$.
Hence $P^2=P$.
It is easy to see that for any $\alpha\ne -1$
\begin{equation*}
({\rm Id}+\alpha P)\Big({\rm Id}-\frac{\alpha}{1+\alpha}P\Big)={\rm Id}
\;\;\hbox{and hence}\;\; ({\rm Id}+\alpha P)^{-1}={\rm Id}-\frac{\alpha}{1+\alpha}P.
\end{equation*}
With $\alpha=\|F\|^2$ this implies \eqref{inverse-g}.

Clearly, $({\rm Id}+\alpha P)F=(1+\alpha)F$ and $({\rm Id}+\alpha P)V=V$ for every $V\perp F$.
Therefore, $\det ({\rm Id}+\alpha P)=1+\alpha$ the product of the eigenvalues,
which yields \eqref{determ-g}.
\end{proof}

The {\bf Riemannian measure} on $U\subset M$ is $d\nu=\sqrt{\det g(x)} dx$,
and we have
\begin{equation}\label{integ}
\int_U f(y) d\nu(y)= \int_V f(\phi(x)) \sqrt{\det g(x)} dx.
\end{equation}
In what follows we shall use the abbreviated notation
\begin{equation}\label{def-tilde-f}
\tilde{f}(x):= f\circ \phi(x)= f(\phi(x)).
\end{equation}

For any $f\in C^\infty(U)$ the {\bf gradient} $\nabla f(y) \in T_yM$ at $y= \phi(x)$
is a vector in $\bR^n$ with components
\begin{equation}\label{def:grad}
(\nabla f(y))^{i} = \sum_j g^{ij}(x)\partial_j \tilde{f}(x), \quad 1\le i\le n,
\end{equation}
and
\begin{equation}\label{inner-f-h}
\langle \nabla f(y), \nabla h(y)\rangle_g
= \sum_{i,j} g^{ij}(x)\partial_i \tilde f(x) \partial_j \tilde h(x).
\end{equation}
Hence
$|\nabla f(y)|_g^2:=\langle \nabla f(y), \nabla f(y)\rangle_g$.

In the chart $(U, \phi^{-1})$ from above the {\bf divergence operator} $\diver$ (see \cite[Theorem~3.14]{Grig}) takes the form
\begin{equation}\label{def:div}
\diver \vec{v}= \frac 1{\sqrt{\det g}} \sum_k \partial_k(\sqrt{\det g}  v^k),
\quad \vec{v}= (v^1,\dots,v^n).
\end{equation}
As before the Laplace operator is defined by
\begin{equation}\label{def:laplace}
\Delta f :=\diver (\nabla f).
\end{equation}

\smallskip

\noindent
{\bf Weights.}
We assume that $w>0$ is a $C^\infty(U)$ weight function such that
\begin{equation}\label{weight-U}
\int_U w d\nu = \int_V w(\phi(x)) \sqrt{\det g(x)} dx <\infty.
\end{equation}
Denote
\begin{equation}\label{weight-V}
\ww(x) :=  w(\phi(x)) \sqrt{\det g(x)}=\tw(x) \sqrt{\det g(x)},
\quad x\in V,
\end{equation}
where just as in \eqref{def-tilde-f}
$\tw(x):=w(\phi(x))$.
Hence, changing the variables leads to
\begin{equation}\label{int-int}
\int_U f(y)w(y) d\nu(y) = \int_V \tilde{f}(x)\ww(x) dx.
\end{equation}
We define the weighted measure $d\nu_w$ on $U$ by
\begin{equation}\label{def-w-mes}
d\nu_w:=wd\nu.
\end{equation}
The {\bf weighted divergence and Laplacian} are defined by (see \cite{Grig}, \S~3.6)
\begin{equation}\label{w-divergence}
\diver_w \vec{v}:= \frac{1}{w}\diver(w\vec{v})
\end{equation}
and
\begin{equation}\label{def:w-laplace}
\Delta_w f := \diver_w (\nabla f)
=\frac{1}{w} \diver (w\nabla f), \quad f\in C^\infty(U).
\end{equation}
In local coordinates the weighted Laplacian takes the form
\begin{align}\label{w-laplace}
\Delta_w f(y)
&= \frac {1}{\tilde w(x)\sqrt{\det g(x)}}\sum_{i=1}^n  \partial_i\Big[\sqrt{\det g(x)} \tilde w(x)
\sum_{j=1}^n g^{ij}(x) \partial_j\tilde f(x)\Big]\notag
\\
&= \sum_{i=1}^n \partial_i \log \big[\sqrt{\det g(x)}\tilde w(x)\big] \sum_{j=1}^n g^{ij}(x) \partial_j \tilde f(x)
+  \sum_{i=1}^n \partial_i \Big[\sum_{j=1}^n g^{ij}(x) \partial_j \tilde f(x)\Big]\notag
\\
& = \sum_{i,j}g^{ij}(x)\partial_i\partial_j \tilde{f}
+ \sum_j \big(\sum_i  \partial_i g^{ij}(x)\big) \partial_j \tilde{f}
\\
& \hspace{1.5in}+  \sum_j \big(\sum_i g^{ij}(x) \partial_i \log \big[\sqrt{\det g(x)}\tilde w(x)\big] \big)\partial_j \tilde{f},\notag
\end{align}
where $\tw(x):=w(\phi(x))$, $y=\phi(x)$, $x\in V$.
We shall denote by $\tilde{\Delta}_w\tilde{f}(x)$ the operator in the right-hand side of \eqref{w-laplace},
i.e. we have
\begin{equation}\label{w-laplace-2}
\Delta_w f(y)=\tilde{\Delta}_w\tilde{f}(x), \quad y=\phi(x), \; x\in V.
\end{equation}

Denote by $C(U)$ the space of continuous functions on $U$ and
by $C_c(U)$ the space of all functions $f\in C(M)$ with compact support contained in $U$.
Also, denote
\begin{equation}\label{def-DU}
\cD(U):= C^\infty(U) \cap C_c(U)
\end{equation}
Further, we denote by
$\overrightarrow{C}^\infty(U)$ the space of smooth vector fields $\vec{v}(x)\in T_xU$
and by
$\overrightarrow{\cD}(U)$ the space of all $\vec{v} \in \overrightarrow{C}^\infty(U)$ with compact support,
contained in $U$.

The weighted divergence theorem \cite[(3.42)]{Grig} takes the form:
If $u\in \cD(U)$ and $\vec{v}\in\overrightarrow{C}^\infty(U)$
or
$u\in C^\infty(U)$ and $\vec{v}\in \overrightarrow{\cD}(U)$,
then
\begin{equation}\label{divergence-w}
\int_U u\diver_w \vec{v} d\nu_w = -\int_U \langle \vec{v}, \nabla u\rangle_g d\nu_w.
\end{equation}
Green's formula remains valid \cite[(3.43)]{Grig}:
If $f, h\in C^\infty(U)$ and $f\in \cD(U)$ or $h\in \cD(U)$, then
\begin{equation}\label{diverge-w}
\int_U f\Delta_w h d\nu_w = -\int_U \langle\nabla f, \nabla h \rangle_g d\nu_w = \int_U h\Delta_w f d\nu_w.
\end{equation}

%%%%%%%%%%%%%%%%%%%%

\noindent
{\bf Neumann extension of the weighted Laplace operator.}
We next describe the Neumann self-adjoint extension $\Delta_w^N$ of the weighted Laplace operator $\Delta_w$ on~$U$.

We consider the operator $\Delta_w$ with domain $\cD(U)$ (see \eqref{w-laplace-2}-\eqref{def-DU}) that is dense in $L^2(U, \nu_w)$.
We denote by $\Delta_w^*$ the adjoint of the operator $\Delta_w$. % with domain $D(\Delta_w):=\cD(U)$.
By \eqref{diverge-w} it readily follow that $\Delta_w$ is symmetric and $-\Delta_w$ is positive.
Therefore, $\Delta_w^*$ is a~closed operator
and $\Delta_w \subset \Delta_w^*$.

It is readily seen that if $f\in D(\Delta_w^*)\cap C^\infty(U)$, then $ \Delta_w^*f= \Delta_wf$.

To define the Neumann extension $\Delta_w^N$ of $\Delta_w$
we introduce a quadratic form $\cE_w^N$ with domain
\begin{equation*}
D(\cE_w^N):=\Big\{f\in L^2(U,\nu_w) \cap C^\infty(U): \int_U|\nabla f|_g^2 d\nu_w <\infty \Big\},
\end{equation*}
defined by
\begin{equation*}
\cE_w^N(f,h):= \int_U \langle \nabla f, \nabla h \rangle_g d\nu_w,
\quad
f, h\in D(\cE_w^N).
\end{equation*}
We also introduce the associated norm
\begin{equation*}
\|f\|_{\cE_w^N}^2:= \|f\|_{L^2}^2 + \cE_w^N(f,f),
\quad f\in D(\cE_w^N).
\end{equation*}

We next show that the symmetric quadratic form $\cE_w^N$ is closable.
Indeed, let $\{f_k\} \subset D(\cE_w^N)$ be such that
$\|f_k\|_2 \to 0$ and $\nabla f_k \mapsto \vec{H}$ in $\overrightarrow{L}^2(U, \nu_w)$.
Then by \eqref{divergence-w}
\begin{equation*}
\int_U\langle \nabla f_k, \vec{v}\rangle_g d\nu_w = -\int_U f_k\diver_w \vec{v} d\nu_w,
\quad \forall \vec{v}\in \vec{\cD}(U).
\end{equation*}
From this and the above assumptions it readily follows that
$\int_U\langle \vec{H}, \vec{v}\rangle_g d\nu_w=0$ for all $\vec{v}\in \vec{\cD}(U)$,
which implies $\vec{H}=\vec{0}$.
Clearly, the above implies that every Cauchy sequence in $D(\cE_w^N)$ is convergent.
Therefore, $\cE_w^N$ is closable.

We denote by $\overline{\cE_w^N}$ the closure of $\cE_w^N$ and by $W_w^N:=D(\overline{\cE_w^N})$ its domain.

Denote
\begin{equation}\label{def-H}
\cH_w := \big\{ f \in L^2(U, \nu_w) \cap C^\infty(U) \cap L^\infty(U): |\nabla f|_g\in L^2(U, \nu_w)\big\}.
\end{equation}

\begin{proposition}\label{prop:dense-algebra}
The set $\cH_w$ is a dense subspace of $W_w^N$ and an algebra.
\end{proposition}
\begin{proof}
Let $ f \in D(\cE_w^N)$.
Choose $\Phi_k \in C^\infty (\R)$ so that
$$
0 \le \Phi'_k \le 1, \quad \Phi_k(x)= x, \;\;\forall x\in[-k,k],
\quad\hbox{and}\quad \|\Phi_k\|_{L^\infty} \le k+1.
$$
By the chain rule
$\nabla (\Phi_k(f))= \Phi'_k(f) \nabla f$
and hence
$\Phi_k(f)\in\cH_w$.
Furthermore, it is readily seen that
\begin{align*}
\int_U |f-\Phi_k(f)|^2 d\nu_w &+ \int_U |\nabla f-\nabla \Phi_k(f)|_g^2d\nu_w
\\
&=\int_M |f-\Phi_k(f)|^2d\nu_w + \int_M |\nabla f-\Phi'_k(f)\nabla(f)|_g^2 d\nu_w  \to 0.
\end{align*}
Therefore, $\cH_w$ is dense in $D(\cE_w^N)$ and hence in $W_w^N$.

\smallskip

To show that $\cH_w$ is an algebra, assume $f,g \in \cH_w$.
As $f,g \in  L^2(U, \nu_w) \cap L^\infty (U)$
it follows that $fg \in L^2(U, \nu_w) \cap L^\infty (U)$.
On the other hand by the product rule
$ \nabla (fg)= f \nabla (g) +g\nabla (f)$ and hence
$|\nabla (fg)|_g \in L^2(U, \nu_w)$.
Therefore, $\cH_w$ is an algebra.
\end{proof}

\begin{definition}\label{def:N-laplace}
We define the domain of the Neumann extension $\Delta_w^N$ of the  weighted Laplacian $\Delta_w$ by
\begin{equation}\label{def-dom-N-laplace}
D(\Delta_w^N) := \big\{f\in W_w^N: \big|\overline{\cE_w^N} (f,\theta)\big| \le c \|\theta\|_2, \;\;\forall \theta \in \cD(U) \big\},
\end{equation}
and for any $f \in  D(\Delta_w^N)$ we define $\Delta_w^N f$ from
\begin{equation}\label{N-laplace}
\int_U \theta\Delta_w^N f d\nu_w = -\overline{\cE_w^N}(f,\theta), \quad \forall \theta\in \cD(U).
\end{equation}
\end{definition}

\begin{proposition}\label{prop:prop-Delta}
The operator $\Delta_w^N$ is self-adjoint and
\begin{equation}\label{prop-delta}
\Delta_w\subset \Delta_w^N\subset \Delta_w^*.
\end{equation}
Moreover,
\begin{equation}\label{domain-laplace-ext}
D(\Delta_w^N):= W_w^N\cap D(\Delta_w^*).
\end{equation}
\end{proposition}
\begin{proof}
From the general theory of positive symmetric quadratic forms (see e.g. \cite[\S 1.3]{FUKU}) it follows that
$\Delta_w^N$ is self-adjoint, i.e. $(\Delta_w^N)^* = \Delta_w^N$.
Also it is easy to se that $\Delta_w\subset \Delta_w^N$.
Hence, $\Delta_w^N=(\Delta_w^N)^*\subset \Delta_w^*$.
Thus \eqref{prop-delta} is valid.

We now prove \eqref{domain-laplace-ext}.
Clearly, \eqref{prop-delta} implies $D(\Delta_w^N)\subset W_w^N\cap D(\Delta_w^*)$.
Let $f\in W_w^N\cap D(\Delta_w^*)$.
Then there exits $\{f_k\}\subset D(\cE_w^N)$ such that
$f_k\to f$ in $L^2(U, \nu_w)$ and $\cE_w^N(f_k, \theta)\to \overline{\cE_w^N}(f_k, \theta)$, $\forall\theta\in \cD(U)$.
From this it follows that for any $\theta\in \cD(U)$
\begin{align*}
\overline{\cE_w^N}(f, \theta)
&= \lim_{n\to\infty} \cE_w^N(f_k, \theta) = \lim_{n\to\infty}\int_U\langle \nabla f_k, \theta \rangle_g d\nu_w
\\
&= -\lim_{n\to\infty}\int_U f_k\Delta_w\theta d\nu_w
= -\int_U f\Delta_w\theta d\nu_w
= -\int_U \theta\Delta_w^*f d\nu_w,
\end{align*}
where we used \eqref{diverge-w}.
From above and \eqref{N-laplace} we infer that
$\Delta_w^*f=\Delta_w^Nf$,
which implies $f\in D(\Delta_w^N)$.
The proof of \eqref{domain-laplace-ext} is complete.
\end{proof}

\subsubsection{The theory of Gyrya and Saloff-Coste}\label{subsec:GS}

The proof of our main result in this section (Theorem~\ref{thm:main})
will rely on a result of Gyrya and Saloff-Coste from \cite{GS}.
To state this result we need the definition
of an {\em inner uniform domain}, which we adapt to the case of Riemannian manifolds.

\begin{definition}\label{def:inner-dom}
Let $U$ be an open connected subset of a Riemannian manifold $(M, d, \nu)$.
The intrinsic distance $d_U(\cdot, \cdot)$ is defined by
\begin{equation}\label{def-d-U}
d_U(y, y_\star):= \inf\big\{\ell(\gamma): \gamma:[0, 1]\mapsto U, \gamma(0)=y, \gamma(1)=y_\star\big\},
\end{equation}
where the curve $\gamma$ is continuous and rectifiable and $\ell(\gamma)$ is its length.

We say that $U$ is an {\em inner uniform domain} if
there exist constants $C, c>0$ such that for any $y,y_\star \in U$
there exists a~rectifiable curve $\gamma: [0,1]\to U$ connecting $y$ and $y_\star$
of length $\le C d_U(y,y_\star)$ such that
$$
d_U(z,\partial U) \ge c d_U(y,z) \wedge d_U(z,y_\star),
\quad
\forall z \in \gamma([0, 1]).
$$
\end{definition}

\begin{remark}\label{rem:inner-u-d}
Observe that if $U$ is convex then the intrinsic distance $d_U(\cdot, \cdot)$ is simply the geodesic distance
inherited from $M$.
One of the important points in this paper is that
every open convex relatively compact subset of $M$
is an inner uniform domain in the sense of Definition~\ref{def:inner-dom}.
This fact (and more) will be established in Theorem~\ref{thm:convex} below.
\end{remark}

We are now prepared to state the result of Gyrya and Saloff-Coste \cite[Theorem~3.34]{GS}.

\begin{theorem}\label{thm:GS}
Let $(M, d, \nu)$ be a complete Riemannian manifold, where the doubling property of the measure \eqref{doubling-0}
and the local Poincar\'{e} inequality \eqref{poincare} are verified.
%For example, these conditions are satisfied if $M$ is compact.
%
Let $U\subset M$ be an inner uniform domain in the sense of Definition~\ref{def:inner-dom}.
Let $d_U(\cdot, \cdot)$ be the intrinsic distance on $U$ extended continuously to $\overline{U}$ $($see $\eqref{def-d-U}$$)$.
Denote
$B_U(y, r):=\{y_\star\in U: d_U(y, y_\star)<r\}$.

Further, assume that $\omega \in C^\infty(U)$ is a weight function
such that $\omega(y)>0$ on $U$, %$\int_U \omega(x) d\mu(x) <\infty$,
and
there exist constants $c>0$ and $N \geq 1$ such that
\begin{equation}\label{oscilate}
\sup_{y_\star \in B_U(y,r)} w(y_\star) \le c \inf_{y_\star \in B_U(y,r)}w(y_\star),
\quad\hbox{$\forall y \in U$, $\forall r>0\;$ so that $d_U(y, \partial U) \ge Nr$.}
 \end{equation}
Set $d\nu_w:= wd\nu$.

Assume also that there exists a constant $c_0>0$ such that
\begin{equation}\label{double}
V_{U,w}(y, 2r) \le c_0 V_{U,w}(y, r),
\quad \forall y \in \overline{U}, \;  \forall r>0,
\end{equation}
where $V_{U,w}(y, r):= \nu_w(B_U(y, r))$.

Let $\Delta_w^N$ be the Neumann extension of the weighted Laplacian $\Delta_w$ from Definition~\ref{def:N-laplace}
and let $e^{t\Delta^N_w}$, $t>0$, be the semi-group generated by $\Delta_w^N$.

Then the respective local Poincar\'{e} inequality is verified
and as a consequence $e^{t\Delta^N_w}$ is an integral operator with (heat) kernel $e^{t\Delta^N_w}(x, y)$
possessing two-sided Gaussian bounds, i.e.
there exist constants $c_1, c_2, c_3, c_4 >0$ such that for $x,y \in \overline{U}$ and $t>0$
\begin{equation}\label{gauss-GS}
\frac{c_1\exp\{- \frac{d_U(x,y)^2}{c_2t}\}}{\big[V_{U,w}(x, \sqrt t) V_{U,w}(y, \sqrt t)\big]^{1/2}}
\le  e^{t\Delta^N_w}(x,y)
\le \frac{c_3\exp\{- \frac{d_U(x,y)^2}{c_4t}\}}{\big[V_{U,w}(x, \sqrt t) B_{U,w}(y, \sqrt t)\big]^{1/2}}.
\end{equation}
\end{theorem}

\subsection{Setting and main result}\label{subsubsec:setting}

Our setting contains two distinctive but closely interconnected parts:

(i) It will be assumed that there exists a symmetric differential operator $L$ acting on functions
defined on a relatively compact open subset $V\subset \R^n$
with polynomial eigenfunctions.

(ii) It will be also assumed that the operator $L$ is a realization in local coordinates of
a weighted Laplace operator $\Delta_w$, acting on functions defined on
a relatively compact open convex subset $U$ of a complete Riemannian manifold $M$
for which the doubling property and the Poincar\'{e} inequality are verified.
The role of the second, geometric part, of our assumption will be critical.

We next present the details of our setting.

\smallskip

\noindent
{\bf Differential operator preserving polynomials on %relatively compact
open set in $\bR^n$.}
Assume that $V\subset \bR^n$ is a connected open set in $\bR^n$ with the properties:
\begin{enumerate}
\item
$X:=\overline{V}$ is compact,
\item
$\mathring{X}=V$, and
\item
$X \setminus V$ is of Lebesgue measure $0$.
\end{enumerate}
Denote by $\tcP_k:=\tcP_k(V)$ the set of all real algebraic polynomials of degree $\le k$ in $n$ variables,
restriction to $V$,
and set $\tcP=\tcP(V):=\cup_{k\ge 0} \tcP_k$.

Let $L$ be a differential operator of the form
\begin{equation}\label{def-L}
L= \sum_{i,j=1}^n a_{ij}(x) \partial_i\partial_j + \sum_{j=1}^n b_j(x)\partial_j,
\end{equation}
where $a_{ij}\in \tcP_2(V)$ and $b_j \in \tcP_1(V)$.
We assume that the domain of the operator $L$ is $D(L):=\tcP(V)$.
Clearly,
\begin{equation}\label{LP}
L(\tcP_k) \subset \tcP_k, \;\; \forall k\ge 0,
\quad\hbox{and}\quad L1=0.
\end{equation}
In addition, we assume that
\begin{equation}\label{LPP}
L(\tcP_k) = \tcP_k, \;\; \forall k\ge 1.
\end{equation}

We also introduce an underlying weighted space $L^2(V, \mu)$, where
\begin{equation}\label{weight-w}
d\mu(x):=\ww(x) dx,
\quad\hbox{where}\;\;
\ww\in C^\infty(V), \;\; \ww>0,
\;\;\hbox{and}\;\; \int_V \ww(x) dx <\infty.
\end{equation}

\smallskip

\noindent
{\bf Laplace operator in chart of Riemannian manifold.}
Assume $(M, d, \nu)$ is an $n$-dimensional complete Riemannian manifold (without boundary) and $M\subset \bR^m$.
We also assume that the Riemannian metric on $M$ is induced by the inner product on $\bR^m$.
We adhere to the notation from \S\ref{subsec:heat-ker}.

We stipulate two conditions on $(M, d, \nu)$:

(i) The volume doubling condition \eqref{doubling-0} is valid.

(ii) The Poincar\'{e} inequality \eqref{poincare} holds true.

As was alluded to in \S\ref{subsubsec:riem-manifold} as a consequence of these two conditions
the (heat) kernel $e^{t\Delta^N}(x,y)$ of the semigroup $e^{t\Delta^N}$ generated
by the Neumann (or Dirichlet) extension of the Laplacian $\Delta$ on $M$ possesses two-sided Gaussian bounds \eqref{gauss-main-0}.
Using the terminology from \cite{GS}  $(M, d, \nu)$ equipped with the quadratic form
$\cE^N$ is a {\em Harnack-type Dirichlet space}.

Further, just as in \S\ref{subsec:weighted-laplace}
we assume that $(U, \varphi)$ is a chart on $M$,
where $U$ is a connected open relatively compact subset of $M$ such that
$\varphi$ maps diffeomorphically $U$ onto $V$,
where $V\subset \bR^n$ is the set from above.
We set $\phi:=\varphi^{-1}$.
As before, for any function $f$ on $U$ we denote
\begin{equation}\label{tilde-f-def}
\tilde f(x):=f(\phi(x)).
\end{equation}

As in \S\ref{subsec:weighted-laplace} we denote by $g(x)=(g_{ij}(x))$ the Riemannian tensor (see \eqref{R-tensor})
and by $g^{-1}(x)=(g^{ij}(x))$ its inverse.

Assume $w>0$ is a $C^\infty(U)$ weight function obeying \eqref{weight-U}
and compatible with $\ww$ from \eqref{weight-w} in the following sense:
\begin{equation}\label{weight-ww}
\ww(x) :=  w(\phi(x)) \sqrt{\det g(x)}=\tw(x) \sqrt{\det g(x)},
\quad x\in V,
\end{equation}
where just as in \eqref{tilde-f-def} $\tw(x):=w(\phi(x))$.
We set $\nu_w:= wd\nu$.

The weighted divergence $\diver_w$ and Laplacian $\Delta_w$ are defined as in \eqref{w-divergence}-\eqref{w-laplace}.

\smallskip

We denote $Y:= \overline{U}$.

\smallskip

\noindent
{\bf Distances and balls.}
We assume that the distance $\rho(\cdot, \cdot)$ on $V$ is induced by the geodesic distance $d(\cdot, \cdot)$ on $U$,
that is,
\begin{equation}\label{def-dist-V}
\rho (x, x_\star) := d(y, y_\star), \quad \forall x, x_\star\in V
\quad\hbox{with}\quad
y:=\phi(x), \;\; y_\star:=\phi(x_\star).
\end{equation}
We denote $B_M(a, r):=\{y\in M: d(a, y)<r\}$
and for any $a\in Y$ set
\begin{equation}\label{def-ball-Y}
B_Y(a, r):=\{y\in Y: d(a, y)<r\}= Y\cap B_M(a, r).
\end{equation}
We also set
\begin{equation}\label{def-ball-X}
B_X(b, r):=\{x\in X: \rho(b, x)<r\}, \quad b\in X.
\end{equation}
We shall use the notation
\begin{equation}\label{measure-ball}
V_{w,Y}(y, r):= \nu_w(B_Y(y, r))
\quad\hbox{}and\quad
V_X(x, r):= \mu(B_X(x, r)).
\end{equation}

\smallskip

\noindent
{\bf Polynomials.}
As we have already alluded to above,
$\tcP_k:=\tcP_k(V)$ stands for the set of real algebraic polynomials of degree $\le k$ in $n$ variables, restricted to $V$,
and $\tcP=\tcP(V):=\cup_{k\ge 0} \tcP_k$.
We now let
\begin{equation}\label{def-pol-Y}
\cP_k(U):=\{f\in C^\infty(U): \tilde{f}\in \tcP_k(V) \}
\quad\hbox{and set}\quad
\cP(U):= \cup_{k\ge 0}\cP_k(U).
\end{equation}

\smallskip

\noindent
{\bf Basic conditions.}
Our further assumptions are as follows:

\smallskip

{\bf C0.}
The operator $L$ from \eqref{def-L} is the weighted Laplacian $\Delta_w$ on $U$ in local coordinates (see \eqref{w-laplace}), i.e.
\begin{align}\label{main-assumpt}
Lh(x) = \sum_{i,j}g^{ij}(x)\partial_i\partial_jh
& + \sum_j \big(\sum_i  \partial_i g^{ij}(x)\big) \partial_j h\notag
\\
& 
+  \sum_j \big(\sum_i g^{ij}(x) \partial_i \log \big[\sqrt{\det g(x)}\tilde w(x)\big] \big)\partial_j h,
\quad x\in V,
\end{align}
or using the notation from \eqref{w-laplace-2} we have
$L\tilde{f}(x)=\tilde{\Delta}_w\tilde{f}(x)$, $x\in V$.

\smallskip

{\bf C1.} The set $U$ is a convex subset of $M$, that is,
for any points $y, y_\star\in U$ there exists a minimizing geodesic line $\gamma\subset U$
that connects $y$ and $y_\star$.

\smallskip

{\bf C2.} {\bf (Doubling property)}
There exists a constant $c_0>0$ such that
\begin{equation}\label{cond:C2}
V_{Y,w}(y, 2r) \le c_0V_{Y,w}(y, r),
\quad \forall y \in Y, \;  \forall r>0,
\end{equation}
or equivalently
\begin{equation}\label{cond:C22}
V_X(x, 2r) \le c_0V_X(x, r),
\quad \forall x \in X, \;  \forall r>0.
\end{equation}
Here $V_{Y,w}(y, r)$ and $V_X(x, r)$ are the weighted volumes of balls, defined in \eqref{measure-ball}.

\smallskip

{\bf C3.}
There exist constants $c>0$ and $N > 1$ such that
\begin{equation}\label{cond:C3}
\sup_{y' \in B_Y(y,r)} w(y') \le c \inf_{y' \in B_Y(y,r)}w(y'),
\quad\hbox{$\forall y \in U$, $\forall r>0\;$ s.t. $d(y, \partial U) \ge Nr$}
 \end{equation}
or equivalently
\begin{equation}\label{cond:C33}
\sup_{x' \in B_X(x,r)} \ww(x') \le c \inf_{x' \in B_X(x,r)}\ww(x'),
\quad\hbox{$\forall x \in V$, $\forall r>0\;$ s.t. $\rho(x, \partial V) \ge Nr$.}
\end{equation}

\smallskip

{\bf C4.}
{\bf (Green's theorem)}
For any $f \in \cP(U)$ and
$h\in L^\infty(U)\cap C^\infty(U)$ such that
$\int_U|\nabla h|_g^2 d\nu_w <\infty$
this identity holds
\begin{equation}\label{divergence1}
\int_U h\Delta_w f d\nu_w= -\int_U \langle\nabla f, \nabla h)\rangle_g d\nu_w, \qquad \hbox{(recall $d\nu_w := wd\nu$)}.
\end{equation}

From \eqref{w-laplace} and \eqref{main-assumpt} it follows that for any $f\in\cP(U)$ we have
$\Delta_wf(y)=L\tilde f(x)$ with $y=\phi(x)$, $x\in V$.
This coupled with the change of variables identity \eqref{int-int} leads to
\begin{equation*}
\int_Uh\Delta_wf d\nu_w = \int_V\tilde h L\tilde fd\mu,
\quad \forall f, h\in\cP(U).
\end{equation*}
In turn this and \eqref{divergence1} yield
that the operator $L$ is symmetric and $-L$ is positive, i.e.
\begin{equation}\label{chance-1}
\int_V hLf d\mu= \int_V fLh d\mu \quad\hbox{and}\quad  -\int_V fLf d\mu \ge 0,
\quad \forall f,h \in \tcP(V),
\end{equation}

Let $\tcV_k:=\tcV_k(X)$ be the orthogonal compliment in $L^2(X,\mu)$ of $\tcP_{k-1}$ to $\tcP_{k}$.
Thus $\tcP_{k}=\tcP_{k-1}\bigoplus\tcV_k$.
By \eqref{LPP} $L(\tilde{\cP}_k) = \tcP_k$.
Hence, due to the symmetry of $L$ we have
$L(\tcV_k) = \tcV_k$.
Since $\tcV_k$ is finite dimensional by the classical theory of symmetric operators
on finite dimensional Hilbert spaces, there exists an orthonormal basis
$\{\tilde P_{kj}:  j=1,\dots, \dim \tcV_k\}$ of $\tcV_k$ consisting of real-valued eigenfunctions (hence polynomials) of $L$.

\smallskip

{\bf C5.}
We assume that there exist eigenvalues
$0=\lambda_0<\lambda_1<\cdots$ such that
\begin{equation}\label{eigenvalues}
L\tilde P_{kj} = -\lambda_k \tilde P_{kj}, \quad j=1,\dots, \dim \tcV_k, \quad k=0, 1, \dots.
\end{equation}

\noindent
{\bf Heat kernel.}
With the assumptions from above it is clear that
\begin{equation*}
\tilde P_k(x, y):=\sum_j \tilde P_{kj}(x)\tilde P_{kj}(y), \quad x, y\in V,
\end{equation*}
is the kernel of the orthogonal projector onto $\tcV_k$.
Then the semigroup $e^{tL}$, $t > 0$, is an integral operator with (heat) kernel $e^{tL}(x, y)$ of the form
\begin{equation}\label{heat-kernel}
e^{tL}(x, y)=\sum_{k=0}^\infty e^{-\lambda_k t}\tilde P_{k}(x, y).
\end{equation}

\begin{remark}\label{rem-polyn}
{\rm (a)}
Observe that the assumption \eqref{LP} is equivalent to requiring
\begin{equation*}
g^{ij}(x)\in \tcP_2(V), \;\;\forall i,j,
\;\;\hbox{and}\;\;
\sum_i g^{ij}(x) \partial_i \log \big[\sqrt{\det g(x)}\tilde w(x)\big] \in \tcP_1(V),
\;\;\forall j.
\end{equation*}

{\rm (b)}
It is important to pointed out that unlike in Green's formula \eqref{diverge-w}
in \eqref{divergence1} it is not assumed that $f$ or $h$ is compactly supported.

\smallskip

{\rm (c)}
In the setting described above we stipulate for convenience that the operator $L$ maps polynomials to polynomials;
this is the case in the particular settings on the ball and simplex.
However, this restriction can be relaxed by replacing the polynomials with other families of functions
in new settings that we anticipate to occur.
\end{remark}

\smallskip

\noindent
{\bf Main general result.}
We now come to one of our principle results.

\begin{theorem}\label{thm:main}
In the setting described above
assume that conditions ${\bf C0-C5}$ are satisfied.
Then the operator $L$ from \eqref{def-L} is essentially self-adjoint and $-L$ is positive.
Moreover, $e^{tL}$, $t>0$, is an integral operator with kernel $e^{tL}(x, y)$ %$p(x, y)$
with Gaussian upper and lower bounds, that is,
there exist constants $c_1, c_2, c_3, c_4 >0$ such that for any $x,y \in X$ and $t>0$
\begin{equation}\label{gauss-main}
\frac{c_1\exp\{- \frac{\rho(x,y)^2}{c_2t}\}}{\big[V_X(x, \sqrt t) V_X(y, \sqrt t)\big]^{1/2}}
\le  e^{tL}(x,y)
\le \frac{c_3\exp\{- \frac{\rho(x,y)^2}{c_4t}\}}{\big[V_X(x, \sqrt t) V_X(y, \sqrt t)\big]^{1/2}}.
\end{equation}
\end{theorem}

\begin{proof}
We shall carry out the proof of Theorem~\ref{thm:main} in several steps.

We first observe that, in our current setting the hypotheses of Theorem~\ref{thm:GS} are satisfied,
in particular, the set $U$ being convex, open and relatively compact is an inner uniform domain (by Theorem~\ref{thm:convex}).
Therefore, $e^{t\Delta_w^N}$, $t>0$, is an integral operator with kernel $e^{t\Delta_w^N}(x, y)$ %$p(x, y)$
with Gaussian upper and lower bounds: For any $x,y \in U$ and $t>0$
\begin{equation}\label{gauss-main-2}
\frac{c_1\exp\{- \frac{d(x,y)^2}{c_2t}\}}{\big[V_{Y, w}(x, \sqrt t) V_{Y, w}(y, \sqrt t)\big]^{1/2}}
\le  e^{t\Delta_w^N}(x,y)
\le \frac{c_3\exp\{- \frac{d(x,y)^2}{c_4t}\}}{\big[V_{Y, w}(x, \sqrt t) V_{Y,w}(y, \sqrt t)\big]^{1/2}}.
\end{equation}

Second, we claim that the operator $L$ is essentially self-adjoint, that is,
the closure $\overline{L}$ of the symmetric operator $L$ is self-adjoint.
Indeed, clearly
\begin{equation*}
D(L)=\Big\{f= \sum_{k,j} a_{kj} \tilde P_{kj}: \; a_{kj}\in\bR, \;\;\{a_{kj}\}\;\;\hbox{compactly supported}\Big\}, \;\;\hbox{and}
\end{equation*}
\begin{equation*}
Lf= -\sum_{k,j} a_{kj}  \lambda_k \tilde P_{kj}
\quad \hbox{if}\quad f= \sum_{j} a_{kj} \tilde P_{kj}\in D(L).
\end{equation*}
We define $\overline{L}$ and $D(\overline{L})$ by
\begin{equation*}
D(\overline{L}):=\Big\{f= \sum_{k=0}^\infty\sum_{j=1}^{\dim \tcV_k} a_{kj} \tilde P_{kj}: \; \sum_{k,j} |a_{kj}|^2 <\infty,\;\;
\sum_{k,j} |a_{kj}|^2 \lambda_k^2<\infty \Big\} %\subset L^2(X, \mu),
\;\;\hbox{and}
\end{equation*}
\begin{equation*}
\overline{L}f:= -\sum_{k,j} a_{kj}  \lambda_k \tilde P_{kj}
\quad \hbox{if}\quad f= \sum_{k, j} a_{kj} \tilde P_{kj}\in D(\overline{L}).
\end{equation*}
One easily shows that $\overline{L}$ is the closure of $L$ and that $\overline{L}$ is self-adjoint.

Third, consider the weighted Laplace operator $\Delta_w$, defined in \eqref{def:w-laplace},
with domain $D(\Delta_w):=\cP(U)$.
As already alluded to above
\eqref{w-laplace} and condition {\bf C0} imply that for any $f\in \cP(U)$
\begin{equation}\label{Delta-L}
\Delta_wf(y)=L\tilde f(x), \quad \hbox{where}\quad y=\phi(x), \;\; \tilde f(x)=f(\phi(x)).
\end{equation}
Denote $P_{kj}(y):= \tilde P_{kj}(\phi^{-1}(y))$ and
$P_k(y, y'):= \sum_j P_{kj}(y)P_{kj}(y')$.
Since $\{\tilde P_{kj}\}$ is an orthonormal basis for $L^2(X, \mu)$, then
$\{P_{kj}\}$ is an orthonormal basis for $L^2(Y, \nu_w)$
and hence $P_k(y, y')$ is the kernel of the orthogonal projector
onto the orthogonal compliment $\cV_k$ of $\cP_{k-1}$ to $\cP_k$ in $L^2(Y, \nu_w)$.
Now, \eqref{eigenvalues} and \eqref{Delta-L} yield
\begin{equation}\label{eigenvalues-2}
\Delta_w P_{kj} = -\lambda_k P_{kj}, \quad j=1,\dots, \dim \cV_k, \quad k=0, 1, \dots.
\end{equation}
Thus there is a complete analogy between the operators $(L, \tilde \cP(X))$ and $(\Delta_w, \cP(Y))$.
As a consequence, $(\Delta_w, \cP(Y))$ is positive and self-adjoint, that is,
the closure $\overline{\Delta}_w$ of $(\Delta_w, \cP(Y))$ in $L^2(Y, \nu_w)$ is self-adjoint.
Then the (heat) kernel $e^{t\overline{\Delta}_w}(y,y')$ of the semi-group $e^{t\overline{\Delta}_w}$
generated by $\overline{\Delta}_w$ takes the form
\begin{equation}\label{kernel-Delta-L}
e^{t\overline{\Delta}_w}(y,y')=\sum_{k=0}^\infty e^{-\lambda_k t}P_{k}(y, y')
\;\;\hbox{and hence}\;\;
e^{t\overline{\Delta}_w}(\phi(x),\phi(x')) = e^{tL}(x, x').
\end{equation}

Clearly, $\cP(U)$ is dense in $\cH_w$, which in turn is dense in $W_w^N$ (see Proposition~\ref{prop:dense-algebra}),
and hence
$\overline{\Delta}_w\subset \Delta_w^N$.
This coupled with the fact that $\overline{\Delta}_w$ and $\Delta_w^N$ are self-adjoint operators
implies $\overline{\Delta}_w=\Delta_w^N$ and hence $e^{t\overline{\Delta}_w} = e^{t\Delta_w^N}$.
Therefore, the two-sided Gaussian bounds in \eqref{gauss-main-2} hold for $e^{t\overline{\Delta}_w}(x,y)$.
This coupled with the right-hand side identity in \eqref{kernel-Delta-L} and \eqref{def-dist-V}
implies \eqref{gauss-main}.
\end{proof}

\subsection{Open relatively compact convex subset of Riemannian manifold}\label{subsec:convex-set}

Here we establish some basic properties of open relatively compact convex subsets of Riemannian manifolds.
In particular, we show that every such set is an inner uniform domain, which was an important ingredient
for the proof of Theorem~\ref{thm:main}.

\begin{theorem}\label{thm:convex}
Let $(M, d, \nu)$ be an $n$-dimensional Riemannian manifold with distance $d(\cdot, \cdot)$ and measure $\nu$.
Let $U$ be an open relatively compact subset of $M$
that is convex in the following sense:
For any $a,b \in U$ there exists a minimizing geodesic line $\gamma \subset U$ connecting $a$ to $b$.
Let $Y:=\overline{U}$ be equipped with the induced metric $d_Y(\cdot, \cdot):=d(\cdot, \cdot)$ and measure $\nu_Y:=\nu$.
As usual for any $a\in Y$ and $R>0$ the ball $B_Y(a,R)$ in the metric space $(Y,d_Y)$
is defined by $B_Y(a,R):= B_M(a,R) \cap Y$, $B_M(a,R):=\{y\in M: d(y, a)<R\}$.
Then:

\smallskip

%\noindent
$(a)$
There exist constants $ 0<c_1 \leq c_2 <\infty$ such that
\begin{equation}\label{m-ball-0}
c_1 R^n \le \nu(B_Y(a,R) )\le c_2 R^n,
\quad \forall a \in Y,\;\; 0<R \le \diam(Y).
\end{equation}

\smallskip

$(b)$
If $\partial U:= Y\setminus U$ is the boundary of $U$, then $\nu(\partial U)= \nu_Y(\partial U)=0$.

\smallskip

$(c)$
$\;\mathring{Y} = U.$

\smallskip

$(d)$
There exist constants $c, C>0$ such that
for any $a,b \in U$ there exists a curve $\gamma \subset  U$ connecting $a$ and $b$ such that
$ \ell(\gamma) \le C d_Y(a,b)$ and
\begin{equation}\label{lower-bound-0}
d(z, U^c) \geq c d(z,a)\wedge d(z,b),\quad \forall z \in \gamma,
\end{equation}
i.e. $U$ is an inner uniform domain in the sense of Definition~\ref{def:inner-dom}.
Here $\ell(\gamma)$ stands for the length of $\gamma$.
\end{theorem}

\subsubsection{Facts from Riemannian geometry}
Here we collect some basic facts from the theory of Riemannian manifolds that will be needed
for the proof of Theorem~\ref{thm:convex}.
We refer the reader to \cite{LANG}, \cite{KLING}, \cite{AUBIN} for more details.

\smallskip

\noindent
{\bf Normal neighbourhood.}
Let $(M, d, \nu)$ be an $n$-dimensional Riemannian manifold.
We shall denote by $|\vec{v}|_g$ the norm of $\vec{v}\in TM$
and by $\|\bar x\|$ the Euclidean norm of $\bar x\in\bR^n$.

We denote by $\Exp$ the exponential map on $M$.
As is well known for any $a \in M$ there exists a constant $R_a >0$ (the injectivity radius) such that
$\Exp_a$ %chart defined on $\bR^n$ %($n=$ dimension of $M$)
maps diffeomorphically the Euclidian ball
$B(0,R_a)\subset \bR^n$ onto $B_M(a,R_a)$ and homeomorphically  $\overline{B(0,R_a)}$ onto  $\overline{B_M(a,R_a)}$.
Furthermore,
\begin{equation}\label{Exp}
\Exp_a 0 = a,
\quad
\Exp_a (B(0,R))= B_M(a,R) \;\; \hbox{for}\;\; 0<R \le R_a.
\end{equation}
We shall term $B_M(a,R_a)$ the {\em normal neighbourhood} of $a\in M$.
Recall the following fundamental properties of $\Exp_a$:
For any $\xi\in \bR^n$ with Euclidean norm $\|\xi\| \le R_a$ the curve
\begin{equation*}
\Big\{\Exp_a (t\xi)\in M: |t|  \le \frac{R_a}{\| \xi\|}\Big\}
\end{equation*}
is geodesic,
and if $-\frac{R_a}{ \|\xi\|} \le t <t' \le \frac{R_a}{ \| \xi\|}$,
then $\{\Exp_a (s\xi)\in M: s\in [t,t']\}$
is the unique minimizing geodesic line connecting  $y:=\Exp_a(t \xi)$ and $y':= \Exp_a(t'\xi)$, and
$d(y,y')= (t'-t )\|\xi\|$.

We shall denote by $g^a(u):=\big(g^a_{ij}(u)\big)$  the metric tensor at $u$ in the $\Exp_a$ chart
$(B_M(a, R_a), \Exp_a^{-1})$.
Note that if  $\|u\| \le R_a$, then
\begin{equation*}
0<  \lambda_a(u) := \inf_{\|\xi\|=1}  \sum_{i,j} g_{ij}^a (u)\xi^i\xi^j
\le  \sup_{\|\xi\|=1}  \sum_{i, j} g_{ij}^a (u)\xi^{i} \xi^j =: \Lambda_a(u).
\end{equation*}
As $u \mapsto \lambda_a(u)$ and $u \mapsto \Lambda_a(u)$ are continuous, by compactness, we have
\begin{equation}\label{def-la-La}
0<\lambda_a :=\inf_{\|u\|\leq R}  \lambda_a(u) \le \sup_{\|u\|\le R}  \Lambda_a(u)=: \Lambda_a <\infty.
\end{equation}
As $g^a(0) ={\rm Id}$ we have $0< \lambda_a \leq 1\leq \Lambda_a <\infty$.

\begin{lemma}\label{lem:Exp}
Let $a \in M$ and assume $\Exp_a$, $R_a$, $\lambda_a$, and $\Lambda_a$ are as above.
Then:

$(i)$
For any measurable function $f: B_M(a,R_a) \mapsto \bR_+$ we have,
using the notation $\tilde{f}(\bar x):= f(\Exp_a(\bar x))$,
\begin{equation}\label{est1}
(\lambda_a)^{\frac n2}\int_{B(0,R_a)} \tilde{f}(\bar x) d\bar x
\le \int_{B_M(a,R_a)} f(x) d\nu(x)
\le  (\Lambda_a)^{\frac n2}\int_{B(0,R_a)} \tilde{f}(\bar x)d\bar x.
\end{equation}
In particular, for any $0<R \le R_a$
\begin{equation}\label{boule}
\frac{\omega_{n-1}}n (\lambda_a)^{\frac n2} R^n
\le \nu(B_M(a,R) )\le \frac{\omega_{n-1}}n  (\Lambda_a)^{\frac n2} R^n,
\quad
\omega_{n-1}:= \frac{2\pi^{n/2}}{\Gamma(n/2)}.
\end{equation}

$(ii)$
If $\bar x, \bar y\in B(0,R_a)$ and $x:=\Exp_a \bar x$, $y:=\Exp_a \bar y$, then $x, y\in B_M(a, R_a)$ and
\begin{equation}\label{dist-norm}
 \sqrt{\lambda_a}  \|\bar x-\bar y\|  \le d(x, y)\le  \sqrt{\Lambda_a}  \|\bar x-\bar y\|.
\end{equation}
%$$\forall \bar x, \bar y \in  B(0,R_a), \quad    \sqrt{\Lambda_a}  |\bar x-\bar y|  \geq d(x,y)\geq  \sqrt{\lambda_a}  |\bar x-\bar y| $$
%where $\|\bar x\|= \Big(\sum_{i=1}^n (x^i)^2\Big)^{1/2}$ is the $n$-dimensional Euclidian norm.
\end{lemma}

\begin{proof}
Estimates \eqref{est1} follow readily by the identity
$$
\int_{B_M(a,R_a)} f(x) d\nu(x)= \int_{B(0,R_a)} \tilde{f}(\bar x)  \sqrt{\det g^a(\bar x)} d \bar x
$$
and the fact that
$$
\det g^a(\bar x) = \prod_{i=1}^n \lambda_i(\bar x) \in[\lambda_a^n, \Lambda_a^n],
$$
where the $\lambda_j(\bar x)$ are the eigenvalues of $\big(g_{ij}^a(\bar x)\big)$.

\smallskip

We now prove part (ii).
Let $\bar x, \bar y \in B(0,R_a)$ and $x:=\Exp_a(\bar x)$, $y:=\Exp_a(\bar y)$.
Set $\bar{\gamma}(t):= t\bar x + (1-t)\bar y$ and $\gamma(t):=\Exp_a(\bar{\gamma}(t))$.
Then
\begin{align*}
d(x,y)&\le \int_0^1 |\gamma' (t)|_g dt
=  \int_0^1\sqrt{\langle (\bar x-\bar y), g^a(\bar{\gamma}(t))(\bar x-\bar y) \rangle} dt\\
& \le \sqrt{\Lambda_a} \int_0^1 \|\bar x-\bar y\| dt =  \sqrt{\Lambda_a}\|\bar x-\bar y\|.
\end{align*}
For the estimate in the other direction,
let $\gamma$ be a minimizing curve connecting $x$ and $y$ and
$\bar{\gamma}(t)=\Exp_a^{-1}(\gamma(t) )$, $\bar{\gamma}(0)=\bar x$, $\bar{\gamma}(1)= \bar y$.
Then similarly as above
$$
d(x,y)= \int_0^1 |\gamma'(t)|_g dt  \ge \sqrt{\lambda_a} \int_0^1 \|\bar x-\bar y\| dt %\int_0^1  \|\bar{\gamma}'(t)\| dt
= \sqrt{\lambda_a}  \|\bar x-\bar y\|.
$$
The above estimates yield \eqref{dist-norm}.
\end{proof}

\begin{lemma}\label{lem:convex1}
Let $B_M(a,R_a)$ be the normal neighborhood of $a\in M$ $($see above$)$ and $0<R \le R_a$.
For $x\in B_M(a,R_a)$ we denote
$\bar x= \Exp_a^{-1}(x)$ and set $x_t:=\Exp_a(t\bar x)$.
Let $U$ be an open convex subset of $M$.
Let $a\in U$ and assume that for some $r>0$ we have
$B_M(x,r) \subset U \cap B_M(a,R)$.
Then
\begin{equation}\label{B-UB}
B_M\big(x_t, tr\qq_a)\subset U\cap B_M(a, tR), \quad 0\le t\le 1,
\;\;\hbox{where}\;\; \qq_a:=\frac{\sqrt{\lambda_a}}{\sqrt{\Lambda_a}},
\end{equation}
and
\begin{equation}\label{dist-Uc}
d(x_t, U^c) \ge \qq_a r d(x_t, a)/R,
%d(x_t, U^c) \ge \qq_a(r/R) d(x_t, a),
\quad 0\le t\le 1.
\end{equation}
Above $\lambda_a$ and $\Lambda_a$ are from \eqref{def-la-La}.
\end{lemma}

\begin{proof}
We begin with the following simple claims:
\begin{equation}\label{BEB1}
B_M\big(x, \sqrt{\lambda_a}\rho\big)\subset \Exp_a\big(B(\bar x, \rho)\big)
\quad\hbox{if}\quad B(\bar x, \rho) \subset B(0, R),
\end{equation}
and
\begin{equation}\label{BEB2}
\Exp_a\big(B(\bar x, \rho/\sqrt{\Lambda_a})\big)\subset B_M(x,\rho)
\quad\hbox{if}\quad B_M(x,\rho) \subset B_a(a, R).
\end{equation}
These two statements follow readily by \eqref{dist-norm}.
Indeed, let  $y\in B_M(x, \sqrt{\lambda_a}\rho)$, i.e. $d(x,y)< \sqrt{\lambda_a}\rho$.
Then using \eqref{dist-norm} we get $\|\bar x-\bar y\|< \rho$, implying $\bar y\in B(0,\rho)$.
Hence, $y\in \Exp_a(B(\bar x, \rho))$, which implies \eqref{BEB1}.
The proof of \eqref{BEB2} is as simple.

\smallskip

We shall use the above to  prove \eqref{B-UB}-\eqref{dist-Uc}.
From $B_M(x,r) \subset B_M(a,R)$, applying \eqref{BEB2} with $\rho=r$, it follows that
$\Exp_a\big(B(\bar x, r/\sqrt{\Lambda_a})\big)\subset B_M(a,R)$
and hence
$B(\bar x, r/\sqrt{\Lambda_a}) \subset B(0,R)$.
We now use the geometry of $\bR^n$ to obtain
$$
B\big(t\bar x, tr/\sqrt{\Lambda_a}\big) \subset B(0,tR), \quad 0\le t\le 1,
$$
and using \eqref{BEB1} we get
$
B_M\big(x_t, tr(\sqrt{\lambda_a}/\sqrt{\Lambda_a})\big) \subset \Exp_a(B(0, tR))=B_M(0,tR).
$

On the other hand, using \eqref{BEB2}, we have
$B\big(\bar x, r/\sqrt{\Lambda_a}\big)\subset \Exp_a^{-1}(B_M(x, r))$.
We now use again the Euclidean geometry of $\bR^n$ to conclude that
\begin{align*}
B\big(t\bar x, tr/\sqrt{\Lambda_a}\big)
&\subset \big\{t\bar y: 0\le t\le 1, \bar y\in B\big(\bar x, r/\sqrt{\Lambda_a}\big) \big\}
\\
&\subset \big\{t\bar y: 0\le t\le 1, \bar y\in \Exp_a^{-1}(B_M(x, r)) \big\}.
\end{align*}
However, for each $\bar y\in \Exp_a^{-1}(B_M(x, r))$ the curve
$\{\Exp_a(t\bar y): 0\le t\le 1\}$
is a geodesic line connecting $a$ and $y\in B_M(x, r)\subset U$
and since $U$ is convex this geodesic line is contained in $U$.
Hence,
$\Exp_a\big(B\big(t\bar x, tr/\sqrt{\Lambda_a}\big)\big) \subset U$.
We now apply \eqref{BEB1} to conclude that
$B_M\big(x_t, tr(\sqrt{\lambda_a}/\sqrt{\Lambda_a})\big)
\subset\Exp_a\big(B\big(t\bar x, tr/\sqrt{\Lambda_a}\big)\big) \subset U$.
Therefore,
$$
B_M\big(x_t, tr\qq_a)
\subset U\cap B_M(a, tR), \quad 0\le t\le 1.
$$
This confirms \eqref{B-UB}.
Now, \eqref{B-UB} implies $d(x_t, U^c)\ge tr\qq_a$.
But $d(x_t, a)=td(x, a)$ and hence
$$
d(x_t, U^c)\ge \frac{d(x_t, a)}{d(x, a)}R\frac{r}{R}\qq_a
\ge \qq_a r d(x_t, a)/R.
$$
The proof of the lemma is complete.
\end{proof}

\smallskip

\noindent
{\bf Uniformisation.}
As is well known (see \cite[Theorem~1.36]{AUBIN}) $R_a$ is continuous as a function of $a\in M$,
and the same is true for $\lambda_a$ and $\Lambda_a$ from \eqref{def-la-La}.
Then taking into account that the set $Y:=\overline{U}$ is compact leads to the conclusion that
the following quantities are well defined:
\begin{equation}\label{def-RX}
R_Y := \min_{a\in Y} R_a >0
\end{equation}
\begin{equation}\label{def-lL}
0 <\lambda := \min_{a \in Y} \lambda_a \le 1 \le  \max_{a\in Y} \Lambda_a =: \Lambda <\infty.
\end{equation}

Now, the following lemma is an immediate consequence of Lemmas~\ref{lem:Exp}, \ref{lem:convex1}.

\begin{lemma}\label{lem:uniform}
$(a)$
If $a\in M$ and $0<R \le R_Y$, then
\begin{equation}\label{ball}
\frac{\omega_{n-1}}n \lambda^{\frac n2} R^n
\le \nu(B_M(a,R) )\le \frac{\omega_{n-1}}n  \Lambda^{\frac n2} R^n,
\quad
\omega_{n-1}:= \frac{2\pi^{n/2}}{\Gamma(n/2)}.
\end{equation}

$(b)$
If $\bar x, \bar y\in B(0,R_Y)$ and $x:=\Exp_a \bar x$, $y:=\Exp_a \bar y$, then $x, y\in B_M(a, R_Y)$ and
\begin{equation}\label{dist-norm-u}
 \sqrt{\lambda}  \|\bar x-\bar y\|  \le d(x, y)\le  \sqrt{\Lambda}  \|\bar x-\bar y\|.
\end{equation}

$(c)$
Let $U$ be an open convex subset of $M$ and $0<R\le R_Y$.
Let $a\in U$ and assume that
$B_M(x,r) \subset U \cap B_M(a,R)$
for some $r>0$.
As before we denote $\bar x= \Exp_a^{-1}(x)$ and set $x_t:=\Exp_a(t\bar x)$.
Then
\begin{equation}\label{B-UB-u}
B_M\big(x_t, tr\qq)\subset U\cap B_M(a, tR), \quad 0\le t\le 1,
\;\;\hbox{where}\;\; \qq:=\frac{\sqrt{\lambda}}{\sqrt{\Lambda}},
\end{equation}
and
\begin{equation}\label{dist-U}
d(x_t, U^c) \ge \qq r d(x_t, a)/R,
\quad 0\le t\le 1.
\end{equation}
\end{lemma}

We next derive from Lemma~\ref{lem:uniform} the following

\begin{lemma}\label{lem:two-point}
Let $U\subset M$ be an open convex set and $a, b\in U$. Let $0<R\le R_Y$.
Assume
$B_M(a,r) \subset U \cap B_M(b,R)$ and $ B_M(b,r) \subset U \cap B_M(a,R)$
and let $\gamma(t)$, $0\le t \le 1$, be a minimizing geodesic line connecting $a$ to~$b$.
Then
\begin{equation}\label{low-bound}
d(\gamma(t), U^c)\ge \qq r d(a, b)/R,
\quad 0\le t\le 1.
\end{equation}
\end{lemma}

\begin{proof}
Under the assumptions of the lemma let $\gamma(t)$, $0\le t \le 1$,
be a minimizing geodesic line connecting $a$ to $b$, i.e. $a=\gamma(0)$, $b=\gamma(1)$.
By \eqref{dist-U} we have
\begin{equation}\label{a-to-b}
d(\gamma(t), U^c)\ge \qq r d(\gamma(t), a)/R,
\quad 0\le t\le 1.
\end{equation}
On the other hand, $\gamma(1-t)$, $0\le t \le 1$, is the geodesic line connecting $b$ to $a$,
and again by \eqref{dist-U} we get
\begin{equation}\label{b-to-a}
d(\gamma(1-t), U^c)\ge \qq r d(\gamma(1-t), b)/R,
\quad 0\le t\le 1.
\end{equation}
Clearly, $d(\gamma(t), a)+ d(\gamma(t), b)= d(a,b)$.
From this and \eqref{a-to-b}-\eqref{b-to-a} we infer that %for $0\le t\le 1$
\begin{equation*}
d(\gamma(t), U^c)\ge \qq r [d(\gamma(t),a) \vee d(\gamma(t), b)]/R
\ge \qq r d(a, b)/2R,
\quad 0\le t\le 1,
\end{equation*}
which confirms \eqref{dist-U}.
\end{proof}

\begin{lemma}\label{lem:uniform-2}
Let $(M,d)$ be a metric space.
Assume that $U\subset M$,
$U\ne \emptyset$, is an open set such that $Y:=\overline{U}$ is compact.
Then for any $R>0$ there exists $r>0$ such that for every $a\in Y$
there exists a ball $B(x_a, r)\subset U\cap B(a, R)$.
\end{lemma}

\begin{proof}
Due to the compactness of $Y$ there exists a finite set of balls $B(a_j, R/2)$, $j=1, \dots, J$,
such that $Y\subset \cup_j B(a_j, R/2)$ and $a_j\in Y$.
Clearly, for each $1\le j\le J$ there exists a ball $B(x_j, r_j)\subset U\cap B(a_j, R/2)$.
Let $r:=\min_{1\le j\le J} r_j$.

We claim that for each $a\in Y$ we have $B(x_j, r)\subset U\cap B(a, R)$ for some $1\le j\le J$.
Indeed, assuming $a\in Y$ we have $a\in B(a_j, R/2)$ for some $1\le j\le J$
and hence  $B(a_j, R/2) \subset B(a, R)$.
Therefore,
$B(x_j, r)\subset U\cap B(a_j, R/2) \subset U\cap B(a, R)$
and this completes the proof.
\end{proof}

The next lemma will be derived from Lemma~\ref{lem:uniform} and Lemma~\ref{lem:uniform-2}.

\begin{lemma}\label{lem:B-UB}
Let $U$ be a convex open subset of $M$ such that $Y= \overline U$ is compact.
Then there exists a constant $\kappa_Y>0$ such that
for any $a\in Y$ and $0<R\le R_Y$ there exists $x\in B_M(a, R)$ such that
\begin{equation}\label{alphaY}
B_M(x,\kappa_Y R) \subset U\cap B_M(a,R)
\end{equation}
and
\begin{equation}\label{lower-bound}
d(x_t, U^c)\ge q\kappa_Y d(x_t, a), \quad 0\le t \le 1.
\end{equation}
Here as before $\bar x := \Exp_a^{-1}(x)$ and $x_t:=\Exp_a(t\bar x)$;
$R_Y$ be the constant from \eqref{def-RX}.
\end{lemma}

\begin{proof}
From Lemma~\ref{lem:uniform-2} it follows that there exists $r_Y>0$ such that
for every $a\in Y$ there exists $y$ such that
$B(y,r_Y)\subset U\cap B(a, R_Y)$.

With $a\in Y$ and $y$ being fixed, denote $\bar y := \Exp_a^{-1}(y)$ and $y_s:=\Exp_a(s\bar x)$.
We apply Lemma~\ref{lem:uniform} to conclude that
\begin{align*}
B_M(y_s, sr_Y q) \subset U\cap B(a, sR_Y), \quad 0\le s\le 1.
\end{align*}
Choose $s$ so that $R= sR_Y$ and set $x:=y_s$. Then from above
\begin{align*}
B_M\Big(x, q\frac{r_Y}{R_Y}R\Big) \subset U\cap B(a, R),
\end{align*}
which implies \eqref{alphaY} with
$\kappa_Y:= qr_Y/R_Y$.

Finally, we apply Lemma~\ref{lem:uniform} to obtain
$$
d(x_t, U^c)\ge \qq^2\frac{r_Y}{R_Y}d(x_t, a) = \qq\kappa_Y d(x_t, a)
\;\; \hbox{for}\;\; 0\le t \le 1,
$$
which confirms \eqref{lower-bound}.
\end{proof}

\subsubsection{Proof of Theorem \ref{thm:convex}}

(a) From \eqref{ball} it follows that there exist constants $C_1, C_2>0$ such that
for any $a\in Y$
\begin{equation}\label{m-ball}
C_1R^n \le \nu(B_M(a, R))\le C_2R^n, \quad 0<R\le R_Y.
\end{equation}

Let $a\in Y$ and $0<R\le \diam Y$.
Two cases present themselves here.

\smallskip

{\em Case 1:} $R\le R_Y$ with $R_Y$ from  \eqref{def-RX}.
By Lemma~\ref{lem:B-UB} there exists $x\in B_M(a, R)$ such that $B_M(x, \kappa_Y R)\subset U\cap B_M(a, R)$.
This and \eqref{m-ball} imply
\begin{equation}\label{m-ball-1}
C_1 \kappa_Y^n R^n \le \nu(B_M(x, \kappa_Y R))
\le \nu(B_Y(a, R)) \le \nu(B_M(R)) \le C_2R^n.
\end{equation}

{\em Case 2:} $R_Y< R\le \diam (Y)$.
Clearly, $\nu(B_Y(a, R)) \le \nu(Y) \le \frac{\nu(Y)}{R_Y^n} R^n$.
On the other hand by Lemma~\ref{lem:B-UB} it follows that there exists $x\in B_M(a, R_Y)$ such that
$B_M(x, \kappa_Y R_Y)\subset U\cap B_M(a, R_Y)$.
This coupled with \eqref{m-ball} leads to
\begin{align*}
\nu(B_Y(a, R)) \ge \nu(B_Y(a, R_Y)) \ge \nu(B_M(x, \kappa_Y R_Y))
\ge C_1 \kappa_Y^n R_Y^n \ge C_1 \frac{\kappa_Y^n R_Y^n}{\diam(Y)^n}R^n.
\end{align*}
Therefore,
$ C_1 \frac{\kappa_Y^n R_Y^n}{\diam(Y)^n}R^n \le \nu(B_Y(a, R))\le \frac{\nu(Y)}{R_Y^n} R^n$.
This and \eqref{m-ball-1} yield \eqref{m-ball-0}.

\medskip

(b)
By \eqref{m-ball-0} it follows that $(Y, d_Y, \nu_Y)$ obeys the doubling property of the measure
and hence it is a homogeneous space.
Therefore, the Lebesgue differentiation theorem is valid.
Then denoting by $\ONE_{\partial U}$ the characteristic function of $\partial U$
we have for almost all $a\in Y$:
\begin{equation}\label{ONE}
\ONE_{\partial U}(a)= \lim_{R \to 0} \frac \ONE{\nu_Y (B_Y(a,R))}\int_{B_Y(a,R)} \ONE_{\partial U} d\nu_Y
= \lim_{R\to 0} \frac{\nu_Y(\partial U \cap B_Y(a,R))}{\nu_Y (B_Y(a,R))}.
\end{equation}
By Lemma~\ref{lem:B-UB} it follows that for any $a\in Y$ and $0<R\le R_Y$
there exists $x_a\in B_Y(a, R)$ such that
$B_Y(x_a, \kappa_Y R)\subset B_Y(a,R)$.
Hence
$$
\nu_Y(\partial U\cap B_Y(a, R)) \le \nu_Y(B_Y(a, R)) - \nu_Y(B_Y(x_a, \kappa_Y R)).
$$
We use this and \eqref{m-ball-0} to obtain
$$
\frac{\nu_Y(\partial U\cap B_Y(a, R))}{\nu_Y(B_Y(a, R))} \le 1- \frac{\nu_Y(B_Y(x_a, \kappa_Y R))}{\nu_Y(B_Y(a, R))}
\le 1- \frac{c_1(\kappa_Y R)^n}{c_2 R^n} = 1-\delta
$$
for some $\delta >0$.
From this and \eqref{ONE} it follows that
$\ONE_{\partial U}(a)\le 1-\delta <1$ for almost all $a\in Y$.
Therefore, $\ONE_{\partial U}(a)=0$ for almost all $a\in Y$,
implying $\nu_Y(\partial U)=0$.

\medskip

(c)
Assume to the contrary that $\mathring{Y} \neq  U$. Hence $\mathring{Y}\setminus U\ne \emptyset$.
Let $a\in \mathring{Y}\setminus U$. %\subset Y\setminus U$.
Then there exists $\eps>0$ such that
$B_M(a, \eps)\subset \mathring{Y}$.
Denote $E:=\mathring{Y}\setminus U\subset \partial U$.
We may assume that $B_M(a, \eps)\subset B_M(a, R_a)$, the normal neighbourhood of $a$ (see \eqref{Exp}).
Then $\Exp_a(B(0, \eps))=B_M(a, \eps)$.

Denote $\tilde{E}:= \Exp_a^{-1}(E\cap B_M(a, \eps))$.
From part (b) of this theorem it follows that
\begin{equation}\label{nu0}
0= \nu(E \cap B_M(a, \eps))
=\int_{B_M(a, \eps)} \ONE_{E}d\nu
\ge c \int_{B(0, \eps)} \ONE_{\tilde{E}}(\bar x)d\bar x.
\end{equation}
We claim that
\begin{equation}\label{ONE-ONE}
\ONE_{\tilde{E}}(\bar x)+\ONE_{\tilde{E}}(-\bar x)\ge 1, \quad \forall \bar x\in B(0, \eps).
\end{equation}
Indeed, if inequality \eqref{ONE-ONE} is not true for some $\bar x\in B(0, \eps)$,
then $\ONE_{\tilde{E}}(\bar x)=0$ and $\ONE_{\tilde{E}}(-\bar x)=0$.
Hence, $x:=\Exp_a \bar x \in U$ and $-x\in U$.
But $U$ is convex and $\{\Exp_a(t\bar x): t\in [-1, 1]\}$ is a geodesic line connecting $x\in U$ and $-x\in U$.
Therefore, it is contained in $U$, in particular, $a=\Exp_a 0\in U$,
which is a contradiction.

Now, we use \eqref{nu0} and \eqref{ONE-ONE} to obtain
\begin{align*}
0\ge c \int_{B(0, \eps)} \ONE_{\tilde{E}}(\bar x)d\bar x
&=c \int_{B(0, \eps)} \ONE_{\tilde{E}}(-\bar x)d\bar x
\\
&= c \int_{B(0, \eps)} \frac{1}{2}(\ONE_{\tilde{E}}(\bar x)d\bar x+ \ONE_{\tilde{E}}(-\bar x)\big) d\bar x
\ge c/2>0.
\end{align*}
This is a contradiction which shows that $\mathring{Y} = U$.

\medskip

(d)
Let $a, b \in U$, $a \ne b$.
We consider two cases depending on whether the distance $d_M(a,b)$ is ``small" or ``large".

\smallskip

{\em Case 1:} $d_M(a,b) \le R_Y$.
Let $\gamma_{a,b}\subset U$ be a minimizing geodesic line connecting $a$ and $b$.
Choose $z \in \gamma_{a,b}$ so that
$R:=d(a,z)= d(z,b)= d(a,b)/2$, $R\le R_Y/2$.
Clearly, $B_M(z,R)\subset B_M(a,2R) \cap B_M(b,2R)$.
Then by Lemma~\ref{lem:B-UB} there exists $c \in B_M(z, R)$ such that
$$
B_M(c, \kappa_Y R) \subset U \cap B_M(z,R) \subset U\cap B_M(a, 2R)\cap B_M(b, 2R).
$$
Note that  $d(a,c) + d(c,b) \le 4R= 2d(a,b)$.

Let $\gamma_{a,c}$ and $\gamma_{c,b}$ be minimizing geodesic lines connecting $a$ to $c$ and $c$ to $b$, respectively.
Let $\gamma$ be the curve $\gamma_{a,c} \cup \gamma_{c,b}$ connecting $a$ and $b$.
For the length $\ell(\gamma)$ of $\gamma$ we have $\ell(\gamma) \le 2d(a,b)$.

We now apply Lemma~\ref{lem:uniform} (c) using that $B_M(c, \kappa_Y R) \subset U \cap B_M(z, 2R)$
to conclude that
\begin{align*}
d(x, U^c) \ge \qq\frac{\kappa_Y R}{2R} d(x, a)
= 2^{-1}\qq\kappa_Y d(x, a),
\quad \forall x\in \gamma_{a,c}
\end{align*}
and similarly we get
\begin{align*}
d(x, U^c) \ge 2^{-1}\qq\kappa_Y d(x, a),
\quad \forall x\in \gamma_{c,b}.
\end{align*}
Therefore,
\begin{align*}
d(x, U^c) \ge c d(x,a) \wedge d(x,b),
\quad \forall x\in \gamma,\;\; c:=2^{-1}\qq\kappa_Y,
\end{align*}
which confirms \eqref{lower-bound-0}.

\smallskip

{\em Case 2:} $d_M(a,b)> R_Y$.
Choose $k \in \bN$, $k \geq 2$, and $R_Y/4<R\le R_Y/2$ so that $kR=d(a, b)$.
Clearly, $k\le 2\diam (Y)/R \le 4 \diam (Y)/R_Y$.

Let $\gamma_{a,b}$ be a minimizing geodesic line connecting $a$ to $b$.
Since $Y$ is convex, then $\gamma_{a,b}\in U$.
Choose points
$a_0, a_1,\dots, a_k \in \gamma_{a,b}$ so that
$a_0=a$, $a_{k}= b$, and $d(a_j,a_{j+1})= R$ for $j=1,\dots,k-1$.
Further, let $b_j\in \gamma_{a, b}$ be the middle point between $a_{j-1}$ and $a_j$,
hence $d(a_{j-1}, b_j)=d(b_j, a_j)$.

By Lemma~\ref{lem:B-UB} there exists $c_j\in B_M(b_j, R/2)$ such that
\begin{equation}\label{B-UBj}
B_M(c_j, \kappa_Y R/2) \subset U\cap B_M(b_j, R/2).
\end{equation}
Let $\gamma\subset U$ be the line connecting $a$ and $b$, obtained as the union of minimizing geodesic lines
$\gamma_{a, c_1}$, $\gamma_{c_j, c_{j+1}}$, $j=1, \dots, k-1$, and $\gamma_{c_k, b}$.
We shall show that the curve $\gamma$ has the stated properties.

Clearly, from \eqref{B-UBj} it follows that  $B_M(c_1, \kappa_YR/2)\subset U\cap B_M(a, R)$.
Applying Lemma~\ref{lem:uniform} (c) we obtain
\begin{equation}\label{lower-bound-a}
d(x, U^c)\ge \qq \frac{\kappa_Y R/2}{R} d(x, a)= 2^{-1}\qq\kappa_Y d(x, a),
\quad \forall x\in \gamma_{a, c_1},
\end{equation}
and similarly
\begin{equation}\label{lower-bound-b}
d(x, U^c)\ge 2^{-1}\qq\kappa_Y d(x, a),
\quad \forall x\in \gamma_{c_k, b}.
\end{equation}
From \eqref{B-UBj} it readily follows that
\begin{align*}
B_M(c_j, \kappa_Y R/2) \subset U\cap B_M(c_{j+1}, 2R)
\;\;\hbox{and}\;\;
B_M(c_{j+1}, \kappa_Y R/2) \subset U\cap B_M(c_j, 2R).
\end{align*}
Also, $B_M(c_j, \kappa_Y R/2)\cap B_M(c_{j+1}, \kappa_Y R/2)=\emptyset$
and hence $d(c_j, c_{j+1})\ge \kappa_Y R/2$.
We now invoke Lemma~\ref{lem:two-point} to conclude that
\begin{equation}\label{lower-bound-j}
d(x, U^c)\ge \qq\frac{\kappa_Y R/2}{2R}d(c_j, c_{j+1})
\ge 2^{-3}\qq\kappa_Y^2 R,
\;\; \forall x\in \gamma_{c_j, c_{j+1}}, j=1, \dots, k-1.
\end{equation}
It is easy to see that $\ell(\gamma)\le 2(k+1)R$ and $k\le 4\diam (Y)/R$,
implying
$R\ge \frac{R_Y\ell(\gamma)}{4\diam(Y)}$.
From this and \eqref{lower-bound-j} we infer that
\begin{equation*}
d(x, U^c)\ge \frac{\qq\kappa_Y^2}{2^5\diam(Y)R_Y}\ell(\gamma)
\ge c d(x,a) \wedge d(x,b),
\;\; \forall x\in \gamma_{c_j, c_{j+1}}, j=1, \dots, k-1,
\end{equation*}
where $c:=\frac{\qq\kappa_Y^2}{2^5\diam(Y)R_Y}$.
This along with \eqref{lower-bound-a} and \eqref{lower-bound-b} implies \eqref{lower-bound-0}.
$\qed$

\subsection{Green's theorem}\label{subsec:diver}

We next establish a general claim that will enable us to verify identity \eqref{divergence1}
(Green's formula)
in particular settings.

\begin{theorem}\label{thm:diver}
Assume that in the setting described in \S \ref{subsubsec:setting}
all conditions are valid but condition {\bf C4}.
Also, assume that there exist sets $V_\eps$, $0<\eps\le 1$, with the following properties:
$V_\eps \subset  \overline{V_\eps} \subset V$,
$ V_\eps\subset V_{\eps'}$ if $0<\eps' <\eps$, and  $\cup_\eps  V_\eps=V$.
Further, assume that the boundary $\partial V_\eps$ of $V_\eps$ is regular
in the sense that
the classical divergence theorem is valid on $V_\eps$:
If $u$ and $\vec{v}$ are a $C^\infty$ function and vector field on $V_\eps$, then
\begin{equation}\label{class-diver}
\int_{V_\eps} u\diver \vec{v}dx
= \int_{\partial V_\eps} u\vec{v}\cdot\vec{n}_\eps d\tau_\eps - \int_{V_\eps} \vec{v}\cdot\nabla u dx,
\end{equation}
where $\vec{n}_\eps$ the unit outward normal to $\partial V_\eps$ vector and
$d\tau_\eps$ is the element of ``area" of $\partial V_\eps$.
Then the identity
\begin{equation}\label{diver-5}
\int_U h\Delta_w f d\nu_w= -\int_U \langle\nabla f, \nabla h)\rangle_g d\nu_w
\end{equation}
holds for all $f\in\cP(U)$ and $h\in C^\infty(U)\cap L^\infty(U)$ such that $\int_U|\nabla h|_g^2d\nu_w<\infty$
if and only if for all such functions
\begin{equation}\label{diver-cond}
\lim_{\eps \to 0}
\int_{ \partial V_\eps} \sum_{i=1}^n\sum_{j=1}^n g^{ij}(x)n_\eps^i(x) \partial_j\tilde{f}(x)\tilde{h}(x) \ww(x) d\tau_\eps(x) =0.
\end{equation}
\end{theorem}

\begin{proof}
Under the hypothesis of the theorem we have, using \eqref{def-w-mes}-\eqref{w-laplace},
\begin{align*}
&\int_U h\Delta_w f d\nu_w = \int_U h\diver(w\Delta f)d\nu
\\
&= \int_V \frac {1}{\sqrt{\det g(x)}}\sum_{i=1}^n  \partial_i\Big[\sqrt{\det g(x)} \tilde w(x)
\sum_{j=1}^n g^{ij}(x) \partial_j\tilde f(x)\Big]
h(\phi(x))\sqrt{\det g(x)} dx
\\
&= \int_V\sum_{i=1}^n  \partial_i\Big[\sqrt{\det g(x)} \tilde w(x)
\sum_{j=1}^n g^{ij}(x) \partial_j\tilde f(x)\Big] h(\phi(x))dx
\\
&= \lim_{\eps\to 0}\int_{V_\eps}\sum_{i=1}^n  \partial_i\Big[\ww(x)
\sum_{j=1}^n g^{ij}(x) \partial_j\tilde f(x)\Big] \tilde{h}(x)dx.
\end{align*}
Now, by the classical divergence theorem \eqref{class-diver}
we obtain
\begin{align*}
\int_{V_\eps}&\sum_{i=1}^n  \partial_i\Big[\ww(x)
\sum_{j=1}^n g^{ij}(x) \partial_j\tilde f(x)\Big] \tilde{h}(x)dx
\\
&=-\int_{V_\eps}\sum_{i=1}^n  \ww(x)\sum_{j=1}^n g^{ij}(x) \partial_j\tilde f(x)\partial_i \tilde{h}(x)dx
\\
&\qquad\qquad\qquad\qquad\qquad\qquad
+ \int_{\partial V_\eps}\sum_{i=1}^n  \ww(x)\sum_{j=1}^n g^{ij}(x) \partial_j\tilde f(x)n_\eps^i(x)\tilde{h}(x)d\tau_\eps(x)
\\
&=-\int_{V_\eps}\sum_{i=1}^n \sum_{j=1}^n g^{ij}(x) \partial_j\tilde f(x)\partial_i \tilde{h}(x)\ww(x)dx
\\
&\qquad\qquad\qquad\qquad\qquad\qquad
+ \int_{\partial V_\eps}\sum_{i=1}^n \sum_{j=1}^n g^{ij}(x) \partial_j\tilde f(x)n_\eps^i(x)\tilde{h}(x)  \ww(x)d\tau_\eps(x)
\\
&= -\int_{U_\eps} \langle\nabla f, \nabla h \rangle_g d\nu_w
+ \int_{\partial V_\eps}\sum_{i=1}^n \sum_{j=1}^n g^{ij}(x) \partial_j\tilde f(x)n_\eps^i(x)\tilde{h}(x)  \ww(x)d\tau_\eps(x).
\end{align*}
Here $U_\eps:=\phi(V_\eps)$
and we used \eqref{weight-V} and \eqref{inner-f-h}.
From the conditions on $f$ and $h$ it readily follows that
$\int_{U_\eps} \langle\nabla f, \nabla h\rangle_g d\nu_w \to \int_{U} \langle\nabla f, \nabla h\rangle_g d\nu_w$
as $\eps\to 0$.
Combining this with the above identities we get the result.
\end{proof}

\section{Heat kernel on the ball}\label{sec:ball}

In this section we establish two-sided Gaussian bounds for the heat kernel generated by the classical operator
\begin{equation}\label{def-L-ball}
L:=\sum_{i=1}^{n}\partial_i^2 - \sum_{i=1}^n \sum_{j=1}^n x_ix_j  \partial_{i}\partial_j  - (n+2\gamma)\sum_{i=1}^{n} x_i\partial_i
\end{equation}
on the unit ball $\bB^n$ in $\bR^n$, $n\ge 1$, equipped with the weighted measure
\begin{equation}\label{mes-ball}
d\mu(x) := (1-\| x\|^2)^{\gamma-1/2} dx,  \quad  \gamma >-1/2,
\end{equation}
and the distance
\begin{equation}\label{dist-ball}
\rho(x,y) := \arccos \big(x\cdot y + \sqrt{1-\| x\|^2}\sqrt{1-\| y\|^2}\big).
\end{equation}
Here we use classical notation for the vectors
$x=(x_1, \dots, x_n)\in \bR^n$,
the inner product
$x\cdot y:=\sum_{j=1}^{n}x_jy_j$,
and the Euclidean norm
$\|x\|:=\sqrt{x\cdot x}$.

We shall use standard notation for balls:
\begin{equation}\label{B-ball}
B(x, r):=\{y\in\bB^n: \rho(x, y)<r\} \quad\hbox{and set}\quad V(x, r):=\mu(B(x, r)).
\end{equation}

Denote by $\tilde{\cP}_k$ the set of all algebraic polynomials of degree $\le k$ in $n$ variables,
and let $\tilde{\cV}_k$ be the orthogonal compliment of $\tilde{\cP}_{k-1}$ to $\tilde{\cP}_k$ in $L^2(\bB^n, \mu)$ when $k\ge 1$.
Then $\tilde{\cP}_k=\tilde{\cP}_{k-1}\bigoplus \tilde{\cV}_k$.
Denote $\tilde{\cV}_0:=\tilde{\cP}_0$.
As is well known (see e.g. \cite[\S2.3.2]{DX}) $\tilde{\cV}_k$, $k=0, 1, \dots$, are eigenspaces of the operator $L$,
more precisely,
\begin{equation}\label{eigen-sp}
L\tilde P=-\lambda_k \tilde P,\quad \forall \tilde P\in\tilde{\cV}_k,
\;\;\hbox{where $\lambda_k:= k(k+n+2\gamma-1)$, $\; k=0, 1, \dots $}.
\end{equation}
Let $\tilde P_{kj}$, $j=1, \dots, \dim \tilde{\cV}_k$, be a real orthonormal basis for $\tilde{\cV}_k$ in $L^2(\bB^n, \mu)$.
Denote $N_k:= \dim \tilde{\cV}_k = \binom{k+n-1}{k}$.
Then
\begin{equation}\label{orth-proj}
\tilde P_k(x, y):=\sum_{j=1}^{N_k}\tilde P_{kj}(x)\tilde P_{kj}(y), \quad x,y\in \bB^n,
\end{equation}
is the kernel of the orthogonal projector onto $\tilde{\cV}_k$.
The heat kernel $e^{tL}(x,y)$, $t>0$, takes the form
\begin{equation}\label{ball-HK}
e^{tL}(x,y)=\sum_{k=0}^\infty e^{-\lambda_k t}\tilde P_k(x, y).
\end{equation}

We consider the operator $L$ defined on  $D(L):= \tilde{\cP}(\bB^n)$
the set of all algebraic polynomials in $n$ variables, restriction to $\bB^n$.
Clearly, $D(L)$ is a dense subset of $L^2(\bB^n, \mu)$.

Here we come to our main result for the heat kernel on the ball:

\begin{theorem}\label{thm:Gauss-ball}
The operator $L$ from \eqref{def-L-ball} in the setting described above
is essentially self-adjoint and $-L$ is positive.
Moreover, $e^{tL}$, $t>0$, is an integral operator whose kernel $e^{tL}(x, y)$ has Gaussian upper and lower bounds, that is,
there exist constants $c_1, c_2, c_3, c_4 >0$ such that for all $x,y \in \bB^d$ and $t>0$
\begin{equation}\label{gauss-ball}
\frac{c_1\exp\{- \frac{\rho(x,y)^2}{c_2t}\}}{\big[V(x, \sqrt t)V(y, \sqrt t)\big]^{1/2}}
\le  e^{tL}(x,y)
\le \frac{c_3\exp\{- \frac{\rho(x,y)^2}{c_4t}\}}{\big[V(x, \sqrt t)V(y, \sqrt t)\big]^{1/2}}.
\end{equation}
\end{theorem}

Before proving this theorem we shall discuss some of its important applications.

\subsection{Smooth functional calculus based on the heat kernel on the ball}\label{subsec:fspp-ball}

As~is shown in \cite{GKPP} smooth functional calculus can be developed in a general setting of Dirichlet spaces
based on the Gaussian bounds of the respective heat kernel.

In our current setting on $\bB^n$, for any bounded function $\Phi$ on $\R$
the operator $\Phi(-L)$ is defined by
\begin{equation*}
\Phi(-L)f:=\sum_{k=0}^\infty \Phi(\lambda_k)\tilde{P}_k f, \quad f\in L^2(\bB^n, \mu),
\end{equation*}
where $\tilde{P}_k$ is the orthogonal projector on $\tilde{V}_k$ with kernel $\tilde{P}_k(x, y)$,
defined in \eqref{orth-proj}.

The upper bound in \eqref{gauss-ball} implies the {\bf finite speed propagation property} (see \cite[Theorem~3.4]{CS}):
There exists a constant $c^\star>0$ such that
\begin{equation}\label{fspp}
\big\langle \cos (t\sqrt{-L})f_1, f_2\big\rangle = 0,
\quad 0<c^\star t<r,
\end{equation}
for all open sets $U_j\subset \bB^n$, $f_j\in L^2(\bB^n, \mu)$, $\supp f_j\subset U_j$, $j=1, 2$,
where $r:=\rho(U_1, U_2)$.

As is shown in \cite[Proposition 2.8]{GKPP} this property implies the following

\begin{proposition}\label{prop:fspp}
Let $\Phi$ be even, $\supp \hat{\Phi}\subset [-A, A]$ for some $A>0$,
and $\hat{\Phi}\in W_1^m$ for some $m>n$, i.e. $\|\hat{\Phi}^{(m)}\|_{L^1} <\infty$.
Here $\hat{\Phi}(\xi):=\int_\R\Phi(u)e^{-iu\xi}du$.
Then for all $x, y\in\bB^n$ and $\delta>0$
\begin{equation}\label{fspp-2}
\Phi(\delta\sqrt{-L})(x, y)=0 \quad\hbox{if}\quad \rho(x, y)> c^\star \delta A.
\end{equation}
Here
$
\Phi(\delta\sqrt{-L})(x, y):= \sum_{k=0}^\infty \Phi(\delta \sqrt{\lambda_k})\tilde{P}_k(x, y).
$
\end{proposition}

Theorem~\ref{thm:Gauss-ball} also implies (see \cite[Theorem 3.7]{CKP}):

\begin{proposition}\label{prop:bound-ker}
If $\Phi$ is a bounded function on $[0, \infty)$ and $\supp \Phi \subset [0, \tau]$, $\tau >0$,
then the kernel $\Phi(\sqrt{-L})(x, y)$ of the operator $\Phi(\sqrt{-L})$ satisfies
\begin{equation}\label{bound-ker}
|\Phi(\sqrt{-L})(x, y)| \le \frac{c\|\Phi\|_\infty}{\big[V(x, \tau^{-1})V(y, \tau^{-1})\big]^{1/2}},
\quad x, y\in \bB^n,
\end{equation}
where $c>0$ is a constant.
\end{proposition}

As is shown in \cite[Theorem 3.1]{GKPP} Propositions \ref{prop:fspp}-\ref{prop:bound-ker}
lead to the following localization result:

\begin{theorem}\label{thm:local}
If $\Phi\in C^m(\bR)$, $m\ge n+1$, is even and $\supp \Phi\subset [-R, R]$, $R>0$,
then the kernel $\Phi(\delta\sqrt{-L})(x, y)$ of the operator $\Phi(\delta\sqrt{-L})$ obeys
\begin{equation}\label{local-ker}
|\Phi(\delta\sqrt{-L})(x, y)|
\le \frac{c_m\big(1+\delta^{-1}\rho(x, y)\big)^{-m}}{\big[V(x, \delta)V(y, \delta)\big]^{1/2}},
\quad x, y\in \bB^n, \; \delta >0,
\end{equation}
where the constant $c_m>0$ depends only on $\|\Phi\|_\infty$, $\|\Phi^{(m)}\|_\infty$, $R$ and $m$.
\end{theorem}

Furthermore, using \cite[Theorem 3.6]{GKPP} the space localization in \eqref{local-ker}
can be improved to sub-exponential by selecting $\Phi\in C^\infty(\bR)$
with ``small derivatives", just as in \cite[Theorem 6.1]{IPX1}.

It should be pointed out in light of the development in \cite{CKP,GKPP} the Gaussian bounds for the heat kernel on $\bB^n$
are the basis for development of Besov and Triebel-Lizorkin spaces on $\bB^n$ and their frame characterization (see \cite{KPX2}),
in the spirit of the development of Frazier and Jawerth \cite{F-J1, F-J2, F-J-W} in the classical case on $\bR^n$.

An important point is that all these results are now valid in the full range of the weight parameter $\gamma > -1/2$ (see \eqref{mes-ball}),
while in \cite{PX2, KPX2} the parameter $\gamma$ is restricted to $\gamma \ge 0$.

\smallskip

In what follows we derive Theorem~\ref{thm:Gauss-ball} as a consequence of Theorem~\ref{thm:main}.

\subsection{Geometric characteristics in a natural chart}\label{subsec:geom-ball}

In the current setting the Riemannian manifold is
$M:=\bS^n:= \{y\in\R^{n+1}: \| y\|=1\}$,
the unit sphere in $\bR^{n+1}$,
%that is, $\bS^n:= \{y\in\R^{n+1}: \| y\|=1\}$,
equipped with the Riemannian metric induced by the inner product on $\bR^{n+1}$.
Denote
\begin{equation*}
V:= \bB^n
\quad\hbox{and}\quad
U:=\bS^n_+= \{y\in\R^{n+1}: \| y\|=1, y_{n+1}>0\}.
\end{equation*}
Clearly, $U=\bS^n_+$ as an open subset of the Riemannian manifold $\bS^n$.
We consider the natural chart $(\bS^n_+, \phi^{-1})$ on $\bS^n$, where the map $\phi: \bB^n \mapsto \bS^n_+$
is defined by
\begin{equation}\label{def-ph}
\phi(x_1, \dots, x_n)
:= \big(x_1, \dots, x_n, \sqrt{1- \| x\|^2}\big).
\end{equation}
In other terms
$$
y_1 = x_1, \dots, y_n =x_n, y_{n+1}= \sqrt{1- \| x\|^2}.
$$
Then
$\phi^{-1}(y_1, \dots, y_{n+1}) = (y_1, \dots, y_n)$.

We equip $\bS^n_+$ and $\bB^n$ with the following weighted measures
\begin{equation}\label{weights}
w(y)d\nu(y):=y_{n+1}^{2\gamma}d\nu(y) %\;\;\hbox{on $\bS^n_+$},
\quad\hbox{and\quad}
\ww(x)dx:=(1-\|x\|^2)^{\gamma-1/2}dx, % \;\;\hbox{on $\bB^n$}
\quad \gamma>-1/2,
\end{equation}
where $\nu$ is the Lebesgue measure on $\bS^n$.
Observe that $d\mu(x)=\ww(x)dx$ is just the measure from \eqref{mes-ball}.

We shall denote by $d(\cdot, \cdot)$ the geodesic distance on $\bS^n$
and by $\rho(\cdot, \cdot)$ the induced distance on $\bB^n$, that is,
$\rho(x,x_\star)=d(\phi(x), \phi(x_\star))$.
It is readily seen that $\rho(\cdot, \cdot)$  is given by \eqref{dist-ball}.
The balls on $\bS^n_+$ will be denoted by $B_Y(y,r)$, namely,
\begin{equation}\label{def-B-Y}
B_Y(y,r):=\{z\in \bS^n_+: d(y, z)<r\}.
\end{equation}

In what follows, just as in \eqref{def-tilde-f} we shall use the abbreviated notation
\begin{equation}\label{tilde-f}
\tilde f(x):= f\circ \phi (x)=f(\phi(x)), \quad x\in\bB^n,
\end{equation}
for a function $f$ defined on $\bS^n_+$.

As in \eqref{R-tensor} the metric tensor (induced by the inner product in $\R^{d+1}$) is given by the matrix
$
g(x)= (g_{ij}(x))
=\big(\big\langle \frac{\partial}{\partial x_i}, \frac{\partial}{\partial x_j} \big\rangle\big).
$
Clearly,
$$
\frac{\partial}{\partial x_i}
=\Big(\frac{\partial y_1}{\partial x_i},\dots, \frac{\partial y_{d+1}}{\partial x_i}\Big)
= \Big(0,\dots, 0, 1, 0, \dots, \frac{-x_i}{\sqrt{1- \| x\|^2}}\Big)
$$
and hence
\begin{equation}\label{g-ij}
g_{ij}(x) = \delta_{ij} + \frac {x_ix_j }{1-\| x\|^2},
\quad 1\le i,j\le n.
\end{equation}
From Proposition~\ref{prop:simpl-map} it follows that the matrix $(g^{ij}(x))$ with entries
\begin{equation}\label{g-ij-1}
g^{ij}(x) := \delta_{ij}-x_ix_j
\end{equation}
is the inverse of $g(x)$, i.e. $g^{-1}(x)=(g^{ij}(x))$.
Appealing again to Proposition~\ref{prop:simpl-map} we infer that
\begin{equation}\label{det-g}
\det g(x)= \frac{1}{1- \|x\|^2}.
\end{equation}

\bigskip

\noindent
{\bf Integration.} Using the above we have
\begin{equation}\label{integral}
\int_{\bS^n_+} f(y) d\nu(y)
= \int_{\bB^n} f(\phi(x))\sqrt{\det g(x)}dx
= \int_{\bB^n} \tilde f(x)\frac{1}{\sqrt{1- \|x\|^2}}dx
\end{equation}
and hence
$$
\int_{\bS^n_+}f(y)w(y)d\nu(y)
= \int_{\bB^n} \tilde f(x) \tilde w(x)\frac{1}{\sqrt{1-\|x\|^2}} dx
= \int_{\bB^n} \tilde f(x)(1-\|x\|^2)^{\gamma-1/2}dx.
$$
In particular,
\begin{align*}
\int_{\bS^n_+}w(y) d\nu(y)
&= \int_{\bB^n} (1- \|x\|^2)^{\gamma-1/2} dx
=|\bS^{n-1}| \int_0^1 (1-r^2)^{\gamma-1/2} r^{n-1}dr\\
&=\frac{|\bS^{d-1}|}2   \int_0^1 (1-v)^{\gamma-1/2} v^{n/2-1}dv
= 2^{-1}B\big(\gamma+1/2, n/2\big)|\bS^{n-1}|.
\end{align*}
Thus,
\begin{equation}\label{BB}
\int_{\bS^n_+} w(y) d\nu(y)
= \int_{\bS^n_+} y_{n+1}^{2\gamma} d\nu(y) = 2^{-1}B\big(\gamma+1/2, n/2\big)|\bS^{n-1}|,
\end{equation}
where $|\bS^{n-1}|=\frac{2 \pi^{n/2}}{\Gamma(\frac n2)}$
is the volume of the unit sphere $\bS^{n-1}$ in $\bR^n$.

\medskip

\noindent
{\bf Representation of \boldmath $\nabla f$ and the weighted Laplacian \boldmath $\Delta_w$ on $\bS^n_+$.}
As in \eqref{def:grad}-\eqref{inner-f-h} we have using \eqref{g-ij-1}
\begin{equation}\label{grad-f}
(\nabla f(y))^{i}
= \sum_{j=1}^n g^{ij}(x) \partial_i \tilde f(x)
= \partial_i \tilde f(x)-\sum_{j=1}^nx_ix_j\partial_j \tilde f(x)
\end{equation}
and
\begin{align}\label{inner-fh}
\langle \nabla f(y), \nabla h(y)\rangle_g
&= \sum_{i,j} g^{ij}(x)\partial_i \tilde f(x) \partial_j \tilde h(x)\notag\\
&=  \sum_{i}\partial_i \tilde f(x) \partial_i \tilde h(x)
-  \sum_{i,j} x_ix_j\partial_i \tilde f(x) \partial_j \tilde h(x).
\end{align}

Also, just as in \eqref{def:w-laplace}-\eqref{w-laplace}
the weighted Laplacian $\Delta_w$ on $\bS^n_+$ is defined by
$\Delta_w f :=\frac{1}{w} \diver (w\nabla f)$ and in local coordinates
\begin{align*}
\Delta_w f(y)
&= \frac {1}{\tilde w(x)\sqrt{\det g(x)}}\sum_{i=1}^n  \partial_i\Big[\sqrt{\det g(x)} \tilde w(x)
\sum_{j=1}^n g^{ij}(x) \partial_j\tilde f(x)\Big]\\
&= \sum_{i=1}^n \partial_i \log [\sqrt{\det g(x)}\tilde w(x)] \sum_{j=1}^n g^{ij}(x) \partial_j \tilde f(x)
+  \sum_{i=1}^n \partial_i \Big[\sum_{j=1}^n g^{ij}(x) \partial_j \tilde f(x)\Big]\\
&= -2(\gamma-1/2) \sum_{i=1}^n  \frac{ x_i}{1- \|x\|^2}
\sum_{j=1}^n (\delta_{ij}-x_ix_j) \partial_j \tilde f(x)\\
& \qquad\qquad\qquad + \sum_{i=1}^n \sum_{j=1}^n \partial_i [g^{ij}(x)\partial_j \tilde f(x)]
=: \cQ_1+\cQ_2,
\end{align*}
where we used that
$\sqrt{\det g(x)}\tilde w(x)= (1-\|x\|^2)^{\gamma-1/2}=\ww(x)$ as $w(y):= y_{n+1}^{2\gamma}$.
Straightforward manipulations show that
$$
\cQ_1=-2(\gamma-1/2)\sum_{j=1}^n x_j\partial_j \tilde f(x) %= -2(\gamma-1/2)x\cdot \nabla \tilde f(x)
$$
and
$$
\cQ_2=\sum_{i=1}^{n}\partial_i^2 \tilde f(x) - \sum_{i=1}^n \sum_{j=1}^n x_ix_j
\partial_i \partial_j \tilde f(x) - (n+1)\sum_{j=1}^{n}x_j\partial_j \tilde f(x). % x\cdot\nabla\tilde f(x).
$$
Therefore, with the notation
$\tilde\Delta_w (\tilde{f})(x):=(\Delta_w f)(\phi(x))$
and $\tilde{f}(x):= f(\phi(x))$ (see (\ref{tilde-f}))
we have for $f \in C^\infty(\bS^n_+)$ (which is the same as $\tilde f \in C^\infty(\bB^n)$)
\begin{equation}\label{laplace-b}
\tilde\Delta_w \tilde{f}(x) = \sum_{i=1}^{n}\partial_i^2 \tilde f(x)
- \sum_{i=1}^n\sum_{j=1}^n x_ix_j  \partial_i \partial_j \tilde{f}(x)
- (2\gamma+n) \sum_{j=1}^{n}x_j\partial_j \tilde f(x)
= L\tilde{f}(x).
 \end{equation}

\subsection{Verification of conditions C0-C5 from \S\ref{subsubsec:setting} and completion of proof}

To apply Theorem~\ref{thm:main} we have to verify conditions {\bf C0-C5} from \S\ref{subsubsec:setting}
in the current setting on $\bB^n$.

%%%%%%% C0

By \eqref{laplace-b} it follows that condition {\bf C0} is obeyed.

%%%%%%% C1

Clearly, $U=\bS^n_+$ is an open and convex subset of $\bS^n$ due to the obvious fact that
the shortest geodesic line connecting any $y, y_\star\in \bS^n_+$ lies in $\bS^n_+$.
Therefore, condition {\bf C1} in \S\ref{subsubsec:setting} is also obeyed.

%%%%%%% C2

Condition {\bf C2} (The doubling property of the measure $d\mu$ on $\bB^n$ or of $wd\nu$ on $\bS^n_+$)
follows readily from the following well known result (see e.g. \cite[Lemma 11.3.6]{DaiXu}):
For any $z\in \bB^n$ and $0<r\le \pi$
\begin{equation}\label{doubling-B}
\int_{B(z,r)} (1-\|x\|^2)^{\gamma-1/2} dx   \sim  r^n (1-\|z\|^2+r^2)^{\gamma}
\end{equation}
or equivalently,
for any $u \in \bS^n_+$ and $0 < r \le \pi$
\begin{equation}\label{doubling-S}
%\int_{B(u,r) \cap \bS^n_+} w(y) d\nu(y) =
\int_{B_Y(u,r)} y_{n+1}^{2\gamma} d\nu(y)   \sim  r^n (u_{n+1}+r)^{2\gamma}.
\end{equation}

%%%%%%% C3

We next verify condition {\bf C3}.
Observe that
if $e_{n+1}=(0,\dots,0, 1)\in \bS^n_+$
is the north pole, then denoting
$$
\theta(y) := d(y, \partial \bS^n_+) = \pi/2- d(y,e_{n+1})
\quad\hbox{for $y\in\bS^n_+$}
$$
we have $y_{n+1}=\sin \theta(y)$.
Assume $y\in\bS^n_+$ and $d(y, \partial \bS^n_+)\ge 2r$, where $0<r\le \pi/4$.
Then, apparently
$\theta(y)-r<\theta(z)<\theta(y)+r$
%$\theta(y)-r<\theta(z)<(\pi/2)\wedge(\theta(y)+r)$
for $z\in B_Y(y, r)$
and hence
\begin{align*}
\frac 1\pi\theta(y)\le\frac 2\pi(\theta(y)-r)\le \sin\big(\theta(y)-r\big)
\le z_{n+1}=\sin \theta(z)\le \theta(y)+r\le 2\theta(y).
\end{align*}
This readily implies
\begin{equation}\label{C3-ball}
\sup_{z \in B_Y(y,r)}  z_{n+1}^{2\gamma} \le (2\pi)^{2|\gamma|}\inf_{z \in B_Y(y,r)} z_{n+1}^{2\gamma},
\end{equation}
which completes the verification of {\bf C3} on $\bS^n_+$.

\smallskip

Similarly as in \eqref{def-pol-Y} we define
\begin{equation*}
\cP_k(\bS^n_+) = \big\{f: f(y_1,\dots,y_{n+1})= P(y_1,\dots,y_n), \; P\in \tilde\cP_k(\bB^n)\big\}
\end{equation*}
and set $\cP(\bS^n_+):= \cup_{k\ge 0} \cP_k(\bS^n_+)$.

A critical step in this development is to establish the following {\bf Green's theorem},
that is the same as to verify condition {\bf C4} in \S \ref{subsubsec:setting}.

\begin{theorem}\label{thm:divergence}
If $f \in \cP(\bS^n_+)$ and $h \in C^\infty(\bS^n_+) \cap L^\infty(\bS^n_+)$
with $\int_{\bS^n_+} |\nabla h|_g^2 wd\nu<\infty$,
then
\begin{equation}\label{divergence}
\int_{\bS^n_+}   h\Delta_w f \,wd\nu
= -\int_{\bS^n_+} \langle \nabla f, \nabla h\rangle_g \,wd\nu.
\end{equation}
\end{theorem}

\begin{proof}
We shall utilize Theorem~\ref{thm:diver} for this proof.

Denote $V_\eps:=\{x\in\R^n: \|x\|^2<1-\eps\}$. Then $\partial V_\eps=\{x\in\R^n: \|x\|^2=1-\eps\}$.
Clearly, $\vec{n}_\eps(x)=\frac{x}{\|x\|}$ is the unit outward normal to $\partial V_\eps$.
We denote by $\tau_\eps$ the Lebesgue measure on the sphere $\partial V_\eps$.
We assume $0<\eps<1/2$.
Appealing to Theorem~\ref{thm:diver} we know that to prove Theorem~\ref{thm:divergence}
we only have to show that
for any $f \in \cP(\bS^n_+)$ and $h \in C^\infty(\bS^n_+) \cap L^\infty(\bS^n_+)$
with $\int_{\bS^n_+} |\nabla h|_g^2 wd\nu<\infty$ we have
\begin{equation}\label{int-Veps2}
J_\eps:=\int_{ \partial V_\eps} \sum_i\sum_j g^{ij}(x)n_\eps^i(x) \partial_j\tilde{f}(x)\tilde{h}(x) \ww(x) d\tau_\eps(x)
\to 0\quad\hbox{as}\;\; \eps\to 0.
\end{equation}
We use \eqref{g-ij-1} and $\vec{n}_\eps(x)=x\|x\|^{-1}$ 
to obtain
\begin{align*}
\sum_i\sum_j g^{ij}(x)n_\eps^i(x) \partial_j\tilde{f}(x)
& = \|x\|^{-1} \sum_i\sum_j x_i(\delta_{ij}-x_ix_j)\partial_j\tilde{f}(x)
\\
& = \|x\|^{-1} \sum_i\Big(x_i\partial_i\tilde{f}(x)-x_i^2\sum_j x_j\partial_j\tilde{f}(x)\Big)
\\
&= \|x\|^{-1}(1-\|x\|^2)\sum_jx_j\partial_j\tilde{f}(x).
\end{align*}
Hence,
\begin{align*}
J_\eps=\int_{\partial V_\eps} \|x\|^{-1}(x\cdot \nabla \tilde f(x)) \tilde{h}(x)(1-\|x\|^2)^{\gamma+1/2} d\tau_\eps(x),
\end{align*}
where $\nabla$ is the standard gradient on $\bR^n$.
Note that $d\tau_\eps=(1-\eps)^{n/2}d\nu$.
Evidently, for any $x\in \partial V_\eps$, $0<\eps<1/2$,
\begin{equation}\label{Jeps-est}
\|x\|^{-1}|x\cdot \nabla \tilde f(x)||\tilde{h}(x)|(1-\|x\|^2)^{\gamma+1/2}
\le \eps^{\gamma+1/2}\|h\|_\infty\sup_{x\in\bB^n}\|\nabla \tilde f(x)\|_\infty.
\end{equation}
However, $\gamma>-1/2$
and $\sup_{x\in\bB^n}\|\nabla \tilde f(x)\|_\infty<\infty$
because $\tilde f$ is a polynomial.
From these and \eqref{Jeps-est} it follows that
$J_\eps\to 0$ as $\eps\to 0$.
\end{proof}

\begin{remark}\label{rem:green-ball}
As one can expect Theorem~\ref{thm:divergence} and \cite[Theorem 2.1]{KPX} are equivalent;
it can be shown that identity $\eqref{divergence}$ can be derived from $(2.3)$ in \cite{KPX}
and vise versa.
\end{remark}

\noindent
{\em Completion of the proof of Theorem~\ref{thm:Gauss-ball}.}
Observe that the current setting on the ball is covered by the setting described in \S\ref{subsubsec:setting}
and conditions {\bf C0-C5} in \S\ref{subsubsec:setting} are verified.
Therefore, Theorem~\ref{thm:Gauss-ball} follows by Theorem~\ref{thm:main}.

\section{Heat kernel on the simplex}\label{sec:simplex}

In this section we establish two-sided Gaussian bounds for the heat kernel generated by the operator
\begin{equation}\label{def-L-simpl}
L := \sum_{i=1}^n x_i\partial_i^2 - \sum_{i=1}^n\sum_{j=1}^n x_ix_j \partial_i\partial_j
+ \sum_{i=1}^n \big(\kappa_i + \tfrac12 - (|\kappa|+ \tfrac{n+1}{2}) x_i\big) \partial_i
\end{equation}
with $|\kappa|:=\kappa_1+\dots+\kappa_{n+1}$
on the simplex
\begin{equation*}
\bT^n:=\Big\{x \in \bR^n: x_1 > 0,\dots, x_n>0,\; |x| < 1 \Big\},
\quad |x|:= x_1+\cdots+x_n,
\end{equation*}
in $\bR^n$, $n\ge 1$, equipped with the measure
\begin{equation}\label{def-mes-simpl}
d\mu(x) = \prod_{i=1}^{n} x_{i}^{\kappa_i-1/2}(1-|x|)^{\kappa_{n+1}-1/2}dx,
\quad  \kappa_i >-1/2,
\end{equation}
and the distance
\begin{equation}\label{def-dist-simpl}
\rho(x,y) = \arccos \Big(\sum_{i=1}^n \sqrt{x_i y_i} + \sqrt{1-|x|}\sqrt{1-|y|}\Big).
\end{equation}

Similarly as before we shall use the notation:
\begin{equation}\label{B-simplex}
B(x, r):=\{y\in\bT^n: \rho(x, y)<r\} \quad\hbox{and}\quad V(x, r):=\mu(B(x, r)).
\end{equation}

Denote by $\tilde{\cP}_k=\tilde{\cP}_k(\bT^n)$ the set of all algebraic polynomials
of degree $\le k$ in $n$ variables restricted to $\bT^n$,
and let $\tilde{\cV}_k=\tilde{\cV}_k(\bT^n)$ be the orthogonal compliment of
$\tilde{\cP}_{k-1}$ to $\tilde{\cP}_k$ in $L^2(\bT^n, \mu)$, $k\ge 1$.
Set $\tilde{\cV}_0:=\tilde{\cP}_0$.
As is well known (e.g. \cite[\S2.3.3]{DX}) $\tilde{\cV}_k$, $k=0, 1, \dots$,
are eigenspaces of the operator $L$, namely,
\begin{equation}\label{sim-eigen-sp}
L\tilde P=-\lambda_k \tilde P,\quad \forall \tilde P\in\tilde{\cV}_k,
\;\;\hbox{where $\lambda_k:= k\big(k+|\kappa|+(n-1)/2\big)$, $\; k=0, 1, \dots$}.
\end{equation}
Let $\tilde P_{kj}$, $j=1, \dots, \dim \tilde{\cV}_k$, be a real orthonormal basis for $\tilde{\cV}_k$ in $L^2(\bT^n, \mu)$.
Denote $N_k:= \dim \tilde{\cV}_k = \binom{k+n-1}{k}$.
Then
\begin{equation}\label{sim-orth-proj}
\tilde P_k(x, y):=\sum_{j=1}^{N_k}\tilde P_{kj}(x)\tilde P_{kj}(y), \quad x,y\in \bT^n,
\end{equation}
is the kernel of the orthogonal projector onto $\tilde{\cV}_k$.
The heat kernel $e^{tL}(x,y)$, $t>0$, takes the form
\begin{equation}\label{simplex-HK}
e^{tL}(x,y)=\sum_{k=0}^\infty e^{-\lambda_k t}\tilde P_k(x, y).
\end{equation}

We consider the operator $L$ with domain $D(L):= \tilde{\cP}(\bT^n):=\cup_{k\ge 0}\tilde{\cP}_k(\bT^n)$
the set of all algebraic polynomials in $n$ variables, restriction to $\bT^n$.
Clearly, $D(L)$ is a dense subset of $L^2(\bT^n, \mu)$.

%%%%%%%%% Theorem

\begin{theorem}\label{thm:Gauss-simpl}
The operator $L$ from \eqref{def-L-simpl} in the setting described above
is essentially self-adjoint and $-L$ is positive in $L^2(\bT^n, \mu)$.
Moreover, $e^{tL}$, $t>0$, is an integral operator with kernel $e^{tL}(x,y)$ with Gaussian upper and lower bounds, that is,
there exist constants $c_1, c_2, c_3, c_4 >0$ such that for any $x,y \in \bT^n$ and $t>0$
\begin{equation}\label{gauss-simplex}
\frac{c_1\exp\{- \frac{\rho(x,y)^2}{c_2t}\}}{\big[V(x, \sqrt t)V(y, \sqrt t)\big]^{1/2}}
\le  e^{tL}(x,y)
\le \frac{c_3\exp\{- \frac{\rho(x,y)^2}{c_4t}\}}{\big[V(x, \sqrt t)V(y, \sqrt t)\big]^{1/2}}.
\end{equation}

\end{theorem}

\begin{remark}\label{rem:simpl}
It would be useful to note that smooth functional calculus on the simplex can be developed using
the two-sided Gaussian bounds on the heat kernel from \eqref{gauss-simplex} and the general results from \cite{CKP, GKPP}.
All comments and results from \S\ref{subsec:fspp-ball} have their analogues for the simplex.
In particular, the finite speed propagation property and Proposition~\ref{prop:fspp} are valid on the simplex as well as
the analogues of the localization estimates from Theorem~\ref{thm:local} and \cite[Theorems 7.1-7.2]{IPX1} hold true.
We shall not elaborate on applications of estimates \eqref{gauss-simplex} any further.
\end{remark}

We shall obtain Theorem~\ref{thm:Gauss-simpl} as a consequence of Theorem~\ref{thm:main}.
We begin by introducing the relevant setting on the simplex.

\subsection{Geometric characteristics in a natural chart}\label{subsec:geom-simpl}

In this setting the Riemannian manifold is again
$M:=\bS^n:= \{y\in\R^{n+1}: \| y\|=1\}$,
the unit sphere in $\bR^{n+1}$,
equipped with the induced Riemannian metric.

There is a natural relationship between $\bT^n$ and the part $\bS^n_T$
of the unit sphere $\bS^n$ in $\R^{n+1}$ lying in the first octant, that is,
$$
\bS^n_T:= \{y\in\bS^n: y_i >0, i=1, \dots, n+1\}.
$$
We shall use the natural chart $(\bS^n_T, \phi^{-1})$ on $\bS^n$, where the map $\phi: \bT^n \mapsto \bS^n_T$
is defined by
\begin{equation}\label{def-ph-simpl}
\phi(x_1, \dots, x_n):=\big(\sqrt{x_1},\dots,\sqrt{x_n}, \sqrt{ 1-|x|}\big), \quad  |x| :=\sum_{i=1}^n x_i,
\end{equation}
or in other terms
$y_i = \sqrt{x_i}, \; i=1,\dots,n$, $y_{n+1}= \sqrt{1-|x|}$.

Then
$\phi^{-1}(y_1, \dots, y_{n+1}) = (y_1^2, \dots, y_n^2)$.

We equip $\bS^n_T$ with the weighted measure
\begin{equation}\label{weight-sim-1}
w(y)d\nu(y):=2^n\prod_{i=1}^{n+1} y_i^{2\kappa_i}d\nu(y),
\quad \kappa_i>-1/2,
\end{equation}
where $d\nu$ is the Lebesgue measure on $\bS^n$,
and $\bT^n$ with
\begin{equation}\label{weight-sim-2}
d\mu(x)=\ww(x)dx := (1-|x|)^{\kappa_{n+1} -\frac 12}\prod_{i=1}^n x_i^{\kappa_i-\frac 12}dx,
\quad \kappa_i>-1/2.
\end{equation}

We shall denote by $d(\cdot, \cdot)$ the geodesic distance on $\bS^n$
and by $\rho(\cdot, \cdot)$ the induced distance on $\bT^n$,
i.e. $\rho(x, x_\star)=d(\phi(x), \phi(x_\star))$.
It is readily seen that $\rho(\cdot, \cdot)$  is given by \eqref{def-dist-simpl}.

As before, for a function $f$ defined on $\bS^n_T$, we shall use the abbreviated notation
\begin{equation}\label{tilde-f-simpl}
\tilde f(x):= f\circ \phi (x)=f(\phi(x)), \quad x\in\bT^n.
\end{equation}

As in \eqref{R-tensor} the metric tensor $g(x)= (g_{ij}(x))$ is given by
$
g_{ij}(x)=\big\langle \frac{\partial}{\partial x_i}, \frac{\partial}{\partial x_j} \big\rangle.
$
Evidently,
$$
\frac{\partial}{\partial x_i}
=\Big(\frac{\partial y_1}{\partial x_i},\dots, \frac{\partial y_{n+1}}{\partial x_i}\Big)
= \Big(0,\dots,0, \frac{1}{2\sqrt{x_i}},0, \dots,0, \frac{1}{2\sqrt{1- |x|}}\Big)
$$
and hence
\begin{equation}\label{gij-T}
g_{ij}(x) = \frac{\delta_{ij}}{4x_i}+\frac{1}{4(1-|x|)}
= \frac{1}{4(1-|x|)}\Big(\frac{\delta_{ij} (1-|x|)}{x_i} + 1\Big).
\end{equation}
A direct verification shows that the matrix with entries
\begin{equation}\label{gij1-T}
g^{ij}(x):= 4(\delta_{ij} x_i- x_ix_j)
\end{equation}
is the inverse to $g(x)$, i.e. $g^{-1}(x)=(g^{ij}(x))$.
We claim that
\begin{equation}\label{det-g-T}
\det g(x) = \frac{4^{-n }}{1-|x|}\prod_{i=1}^n \frac{1}{x_i}.
\end{equation}
This identity follows readily by the following lemma.

%%%%%%% Lemma

\begin{lemma}
Given $(a)=(a_1,\dots,a_n) \in \bR^n$, $n\ge 2$, let
$$
A
:=\left(\begin{array}{cccc}
a_1+1  & 1  & \cdots   & 1  \\
1 & a_2+1  & \cdots  & 1  \\
\vdots & \vdots & \cdots &\vdots\\
1 & 1  & \cdots  & a_n+1
\end{array}\right).
$$
Then
\begin{equation}\label{det-A}
\det A = \prod_{i=1}^n a_i + \sum_{j=1}^n  \prod_{k=1, k\neq j}^n a_k.
\end{equation}
\end{lemma}

\noindent
{\bf Proof.}
Let $e_j$ be the $j$th coordinate vector (column) in $\R^n$, $1\le j\le n$,
and set $\ONE^T:=(1, 1, \dots, 1)$, $\ONE\in \R^n$.
Then we have $A=(a_1e_1+\ONE, a_2e_2+\ONE, \dots, a_ne_n+\ONE)$.
By splitting the first column of $A$ into two we can write
$$
\det A= \det(a_1e_1, a_2e_2+\ONE, \dots, a_ne_n+\ONE)
+ \det(\ONE, a_2e_2+\ONE, \dots, a_ne_n+\ONE).
$$
In the second determinant we subtract the first column from all other columns to obtain
$$
\det(\ONE, a_2e_2+\ONE, \dots, a_ne_n+\ONE) = \det(\ONE, a_2e_2, \dots, a_ne_n).
$$
Precisely in the same way we get
\begin{align*}
\det(a_1e_1, a_2e_2+\ONE, \dots, a_ne_n+\ONE)
&= \det(a_1e_1, a_2e_2, a_3e_3+\ONE,\dots, a_ne_n+\ONE)\\
&+ \det(a_1e_1, \ONE, a_3e_3,\dots, a_ne_n).
\end{align*}
Inductively we obtain
\begin{align*}
\det A &= \det(a_1e_1,\dots, a_ne_n)
+ \sum_{j=1}^n \det(a_1e_1,\dots, a_{j-1}e_{j-1}, \ONE,a_{j+1}e_{j+1},\dots, a_ne_n).
\end{align*}
Obviously $\det(a_1e_1,\dots, a_ne_n)= a_1\dots a_n$ and it is easy to see that
$$
(a_1e_1,\dots, a_{j-1}e_{j-1}, \ONE,a_{j+1}e_{j+1},\dots, a_ne_n)
= \prod_{k=1, k\ne j}^n a_k.
$$
Putting the above together we arrive at (\ref{det-A}).
$\qed$

\medskip

\noindent
{\bf The gradient $\nabla$ and weighted Laplacian $\Delta_w$ on $\bS^n_T$.}
Using the chart $(\bS^n_T, \phi^{-1})$ and \eqref{gij1-T} we obtain for $y=\phi(x)$, $x\in \bT^n$,
$$
(\nabla f (y))^{i}= \sum_jg^{ij}(x) \partial _j f(\phi(x))
= 4x_i\Big[\partial _i\tilde{f}(x) - \sum_{j=1}^n x_j\partial_j\tilde{f}(x)\Big].
$$
Also, we have
\begin{align}\label{grad-grad-T}
\langle \nabla f(y), \nabla  h(y)\rangle_g
&=\sum_{i,j}g^{ij}(x) \partial_i \tilde f(x) \partial_j \tilde h(x)\notag
\\
& = 4  \sum_{i,j}(\delta_{ij} x_i -x_ix_j) \partial_i \tilde f(x) \partial_j \tilde h(x)
\\
%& = 4 \Big[ \sum_{i=1}^n x_i \partial_i \tilde f(x) \partial_i   \tilde h(x)
%-  (x\cdot \nabla \tilde f(x)) (x\cdot \nabla \tilde h(x))\Big].\notag
%\\
& = 4 \Big[ \sum_{i} x_i \partial_i \tilde f(x) \partial_i   \tilde h(x)
-  \sum_{i,j}x_ix_j\partial_i\tilde f(x)\partial_j\tilde h(x)\Big].\notag
\end{align}
As in \eqref{def:w-laplace} the weighted Laplacian $\Delta_w$ is defined by
$\Delta_w f :=\frac{1}{w} \diver (w\nabla f)$
and we set
$\tilde\Delta_w \tilde f:=\Delta_w f (\phi(x))$.
Just as in \eqref{w-laplace} we get
\begin{align*}
\tilde\Delta_w \tilde f(x)
& = \frac{1}{\tilde w(x) \sqrt{\det g(x)}}
\sum_{i=1}^n  \partial_i(\sqrt{\det g(x)}\tilde w(x)
\sum_{j=1}^n g^{ij}(x) \partial_j\tilde f(x)\\
& = \sum_{i=1}^n \partial_i \log \big[\tilde w(x) \sqrt{\det g(x)}\big]
\sum_{j=1}^n g^{ij}(x) \partial_j\tilde f(x)
+  \sum_{i=1}^n \partial_i\Big[\sum_{j=1}^n g^{ij}(x)\partial_j\tilde f(x)\Big]\\
& = \sum_{i=1}^n \Big[\frac{\kappa_i -1/2}{x_i}-\frac {\kappa_{n+1} -1/2}{1-|x|}\Big]
\sum_{j=1}^n g^{ij}(x) \partial_j\tilde f(x)\\
&\qquad + \sum_{i=1}^n
\sum_{j=1}^n \partial_i [g^{ij}(x)]\partial_j\tilde f(x)
+ \sum_{i=1}^n  \sum_{j=1}^n g^{ij}(x)\partial_i \partial_j\tilde f(x)
=:\cQ_1+\cQ_2+\cQ_3.
\end{align*}
Now, using \eqref{gij1-T} we get
\begin{equation}\label{Q1}
\sum_{j=1}^n g^{ij}(x) \partial_j\tilde f(x)
= 4\sum_{j=1 }^n [\delta_{ij}x_i-x_ix_j]\partial_j\tilde f(x)
= 4x_i\Big[\partial_i\tilde f(x)-\sum_{j=1}^n x_j\partial_j\tilde f(x)\Big]
\end{equation}
and hence
\begin{align*}
\frac{1}{4}\cQ_1
&= \sum_{i=1}^n \Big(\frac{\kappa_i -1/2}{x_i}-\frac {\kappa_{n+1} -1/2}{1-|x|}\Big)
x_i\Big[\partial_i\tilde f(x)-\sum_{j=1}^n x_j\partial_j\tilde f(x)\Big]
\\
& = \sum_{i=1}^n (\kappa_i -1/2)\partial_i\tilde f(x)
-\sum_{j=1}^n x_j\partial_j\tilde f(x)\Big(\sum_{i=1}^n \kappa_i -n/2\Big)
\\
&\qquad\qquad\qquad\qquad\qquad
- (\kappa_{n+1} -1/2)\sum_{j=1}^n x_j\partial_j\tilde f(x)
\\
& = \sum_{i=1}^n (\kappa_i -1/2)\partial_i\tilde f(x)
-\big[|\kappa|-(n+1)/2\big]\sum_{j=1}^n x_j\partial_j\tilde f(x).
\end{align*}
Recall that
$|\kappa|:= \kappa_1+\cdots+\kappa_{n+1}$.
By (\ref{gij1-T}) we have
\begin{align*}
\frac{1}{4}\cQ_2
&=-\sum_{i=1}^n
\sum_{j=1, j\ne i}^n x_j\partial_j\tilde f(x)+\sum_{i=1}^n(1-2x_i)\partial_i \tilde f(x)
\\
&= \sum_{i=1}^n\partial_i \tilde f(x)-(n+1)\sum_{j=1}^n x_j\partial_j\tilde f(x)
\end{align*}
and
\begin{align*}
\frac{1}{4}\cQ_3
&=\sum_{i=1}^n  \sum_{j=1}^n (\delta_{ij}x_i-x_ix_j)\partial_i \partial_j\tilde f(x)
=\sum_{i=1}^n x_i\partial_i^2\tilde f(x)
- \sum_{i=1}^n  \sum_{j=1}^n x_ix_j\partial_i \partial_j\tilde f(x).
\end{align*}
Combining the above expressions for $\cQ_1$, $\cQ_2$, and $\cQ_3$ we obtain that
for any function $f\in C^\infty(\bS^n_T)$
\begin{align*}
\frac{1}{4}\tilde\Delta_w \tilde f(x)
&=\sum_{i=1}^n x_i\partial_i^2\tilde f(x)
- \sum_{i=1}^n  \sum_{j=1}^n x_ix_j\partial_i \partial_j\tilde f(x)\\
&+  \sum_{i=1}^n (\kappa_i +1/2)\partial_i\tilde f(x)
-\big[|\kappa|+(n+1)/2\big]\sum_{j=1}^n x_j\partial_j\tilde f(x).
\end{align*}
Hence,
\begin{equation}\label{laplacian-simpl}
\tilde\Delta_w \tilde f(x) = 4L\tilde{f}(x), \quad\forall x\in\bT^n.
\end{equation}

\smallskip

\noindent
{\bf Integration.} Using the chart $(\bS^n_T, \phi^{-1})$ and \eqref{det-g-T} we obtain
$$
\int_{\bS^n_T} f(y) d\nu(y)
= \int_{\bT_n} f(\phi(x)) \sqrt{\det g(x)}dx
= 2^{-n} \int_{\bT^n} \tilde f(x) \prod_{i=1}^n x_i^{ -1/2 } (1-|x|)^{ -1/2 }dx
$$
and hence
\begin{equation}\label{weight-int}
\int_{\bS^n_T}  f(y)  w(y) d\nu(y)
= \int_{\bT^n} \tilde f(x) \prod_{i=1}^n x_i^{\kappa_i-1/2} (1-|x|)^{\kappa_{n+1}-1/2}dx
= \int_{\bT^n} \tilde f(x)\ww(x)dx.
\end{equation}
For $\kappa_i >-1/2$, $j=1,\dots, n+1$, a little calculus gives
\begin{equation}\label{weight-measure-2}
\int_{\bT^n}\prod_{i=1}^n x_i^{\kappa_i-1/2} (1-|x|)^{\kappa_{n+1}-1/2}dx
=\prod_{i=1}^n B\Big(\kappa_i+\frac 12, \sum_{j=i+1}^{n+1}\big(\kappa_j+\frac 12\big)\Big),
\end{equation}
and using \eqref{weight-int} we get
\begin{equation}\label{weight-measure}
2^n\int_{\bS^n_T} \prod_{i=1}^{n+1} y_i^{2\kappa_i} d \nu(y)
= \prod_{i=1}^n B\Big(\kappa_i+\frac 12, \sum_{j=i+1}^{n+1}\big(\kappa_j+\frac 12\big)\Big) <\infty.
\end{equation}
Above $B(\cdot, \cdot)$ stands for the standard beta function.

\subsection{Verification of conditions C0-C5 from \S\ref{subsubsec:setting} and completion of proof}

The proof of Theorem~\ref{thm:Gauss-simpl} relies on Theorem~\ref{thm:main}, which requires
the verification of conditions {\bf C0-C5} from \S\ref{subsubsec:setting}.

From \eqref{laplacian-simpl} it follows that the condition {\bf C0} from \S\ref{subsubsec:setting}
is satisfied for the operator $4L$.
Here the factor $4$ is insignificant because apparently $e^{t4L}=e^{4tL}$
and if Theorem~\ref{thm:Gauss-simpl} holds for the operator $4L$ it holds for $L$.

Clearly, $\bS^n_T$ is an open and convex subset of $\bS^n$ as
the shortest geodesic connecting any $y, y'\in \bS^n_T$ lies in $\bS^n_T$.
Hence condition {\bf C1} is obeyed.

%%%%%%%%%%%% C2

The doubling property of the measure $wd\nu$ is well known,
i.e. condition {\bf C2} is obeyed.
In fact, this is an immediate consequence of the following claim (see e.g. \cite[(5.1.10)]{DaiXu}):
For any $u \in \bS^n_T$ and $0 < r \le \pi/2$
\begin{equation}\label{doubling-T}
\int_{B_Y(u,r)} w(y) d\nu(y) =
\int_{B_Y(u,r)} 2^n\prod_{i=1}^{n+1}y_i^{2\kappa_i} d\nu(y)
\sim  r^n \prod_{i=1}^{n+1}(u_i+r)^{2\kappa_i}
\end{equation}
or equivalently, for any $z\in\bT^n$ and $0<r\le 1$
\begin{equation}\label{doubling-T2}
\int_{B(z,r)}\prod_{i=1}^{n} x_{i}^{\kappa_i-1/2}(1-|x|)^{\kappa_{n+1}-1/2}dx
\sim  r^n (1-|z|+r^2)^{\kappa_{n+1}}\prod_{i=1}^{n}(z_i+r^2)^{\kappa_i}.
\end{equation}

%%%%%%%%%%%% C3

To verify {\bf C3} we need introduce some notation.
The boundary $\partial \bS^n_T$ of $\bS^n_T$ can be represented as
$\partial \bS^n_T=\cup_{i=1}^{n+1} \Gamma_i$, where
\begin{equation}\label{def-Gamma-i}
\Gamma_i:=\{y\in \overline{\bS^n_T}: y_i=0\}.
\end{equation}
Further, for $y\in \bS^n_T$ denote
\begin{equation}\label{def-theta-i1}
\theta_i(y) := \pi/2 -d(y, e_i), \quad i=1,\dots, n+1,
\end{equation}
where $e_i$ is the $i$th coordinate vector in $\R^{n+1}$.
Clearly, $\theta_i(y) = d(y, \Gamma_i)$ and hence
\begin{equation}\label{inf-theta-i}
\inf_{1\leq i \le n+1} \theta_i(y) = d(y, \partial \bS^n_T).
\end{equation}
Note that $y_i=\sin \theta_i(y)$, which implies $y_i \sim \theta_i(y)$.

Assume $y\in\bS^n_T$ and $d(y, \partial \bS^n_T)>2r$ with $0<r\le \pi/4$.
Then from (\ref{inf-theta-i}) it follows that $\theta_i(y) \ge 2r$ for $i=1, \dots, n+1$.
Now, just as in the proof of \eqref{C3-ball} we obtain
$$
\sup_{z\in B_Y(y,r)} z^{2\kappa_i} \le (2\pi)^{4|\kappa_i|}\inf_{z\in B_Y(y,r)} z^{2\kappa_i}
$$
and hence
$$
\sup_{z \in B_Y(y,r)}\prod_{i=1}^{n+1}z_i^{2\kappa_i}
\le  c\inf_{z \in B_Y(y,r)}\prod_{i=1}^{n+1}z_i^{2\kappa_i},
$$
which confirms condition {\bf C3} on $\bS^n_T$.

\smallskip

Recall that $\tilde{\cP}_k(\bT^n)$ is the set of all polynomials of degree
$\le k, $ on $\bR^n,$ restricted to $\bT^n$,
and $\tilde{\cP}(\bT^n) = \cup_{k\ge 0} \tilde{\cP}_k(\bT^n)$.
Let $\cP_k(\bS^n_T)$ and  $\cP(\bS^n_T)$ be the respective spaces on $\bS^n_T$, i.e.
\begin{equation*}
\cP_k(\bS^n_T) := \big\{f: f(y_1, \dots, y_{n+1})=P(y_1^2,\dots,y^2_n ), \; P\in  \tilde{\cP}_k(\bT_n) \big\}
\end{equation*}
and $\cP(\bS^n_T):=\cup_{k\ge 0}\cP_k(\bS^n_T)$.

\smallskip

The following {\bf Green's theorem} plays a critical role here.

\begin{theorem}\label{thm:divergence-T}
If $ f \in \cP (\bS^n_T)$ and $h \in C^\infty(\bS^n_T) \cap L^\infty(\bS^n_T)$ with $\int_{\bS^n_T}|\nabla h|_g^2w d\nu<\infty$,
then
\begin{equation}\label{divergence-T}
\int_{\bS^n_T} h\Delta_w f w d\nu
= - \int_{\bS^n_T} \langle \nabla f, \nabla h\rangle_g w d\nu.
\end{equation}
\end{theorem}

\begin{proof}
This proof will rely on Theorem~\ref{thm:diver}.
Denote $V:=\bT^n$ and let $\partial V$ be its boundary.
We introduce the sets
\begin{equation}\label{def-V-eps}
V_\eps:=\Big\{x\in \bR^n: x_1 > \eps,\dots, x_n>\eps, \; \sum_{i=1}^{n}x_i < 1-\eps\Big\}, \;\; \eps>0.
\end{equation}
The following properties of the sets $V_\eps$ follow immediately from the definition:
$V_\eps\subset V$, $V_\eps\subset V_{\eps'}$ if $0<\eps'<\eps$, and
$\cup_{\eps>0}V_\eps=V$.
Also,
$\partial V_\eps = \cup_{i=1}^n \overline{F}_\eps^i\cup \overline{H}_\eps$,
where
\begin{equation*}
F_\eps^i:=\Big\{x\in \bR^n: x_i=\eps, x_j> \eps \;\hbox{if}\; j\ne i,\; \sum_{j\ne i}x_j< 1-2\eps \Big\}
\end{equation*}
and
\begin{equation*}
H_\eps:=\Big\{x\in \bR^n: x_1> \eps,\dots, x_n>\eps, \; \sum_{j=1}^nx_j=1-\eps \Big\}.
\end{equation*}
The boundary of $\partial V_\eps$ is a polyhedron in $\bR^n$
and hence it is regular,
that is, the classical divergence formula \eqref{class-diver} is valid on $V_\eps$
(see e.g. Theorem~1, \S 5, Chapter~I in \cite{FS}).
Therefore, we can use Theorem~\ref{thm:diver}.

We shall also need the scaled simplex $\bT^n_b$, defined by
\begin{equation*}
\bT^n_b:=\Big\{x\in \bR^n: x_1>0,\dots, x_n>0,\; \sum_{i=1}^{n}x_i <b\Big\}, \quad b>0.
\end{equation*}
By changing the variables it follows from \eqref{weight-measure-2} that
\begin{equation}\label{w-meas-T}
\int_{\bT^n_b}\prod_{i=1}^n x_i^{\kappa_i-1/2} (b-|x|)^{\kappa_{n+1}-1/2}dx
=b^{|\kappa|+(n+1)/2}\prod_{i=1}^n B\Big(\kappa_i+\frac 12, \sum_{j=i+1}^{n+1}\big(\kappa_j+\frac 12\big)\Big)<\infty.
\end{equation}

Let $f$ and $h$ be the functions from the hypothesis of the theorem and
let $\vec{n}_\eps=(n_\eps^1,\dots, n_\eps^n)$ be the unit outward normal vector to $\partial V_\eps$.
Denote
\begin{equation}\label{def-G-eps}
G_\eps(x):= \sum_{i=1}^n\sum_{j=1}^n g^{ij}(x)n_\eps^i(x) \partial_j\tilde{f}(x)\tilde{h}(x) \ww(x),
\quad x\in \partial V_\eps.
\end{equation}
In light of Theorem~\ref{thm:diver} to prove Theorem~\ref{thm:divergence-T} it suffices to show that
\begin{equation}\label{int-G-eps}
\lim_{\eps\to 0}\int_{\partial V_\eps} G_\eps d\tau_\eps = 0,
\end{equation}
where $d\tau_\eps$ is the element of ``area" of $\partial V_\eps$.
Henceforth, we shall assume that $\eps>0$ is sufficiently small, e.g. $\eps<1/(n+1)$.

Denote
\begin{equation}\label{def-Xi}
X_i(x):= \sum_{j=1}^n g^{ij}(x) \partial_j \tilde f(x)
= 4x_i\Big[\partial_i \tilde f(x) - \sum_{j=1}^n x_j\partial_j \tilde{f}(x)\Big], \quad i=1, \dots, n,
\end{equation}
where we used \eqref{gij1-T}.
Then using the notation $\vec{X}(x):= (X_1(x), \dots, X_n(x))$ we have
\begin{equation}\label{rep-G-eps}
G_\eps(x)= \tilde{h}(x)\ww(x)\vec{X}(x)\cdot \vec{n}_\eps(x).
\end{equation}

To estimate $\int_{\partial V_\eps} |G_\eps|d\tau_\eps$
we have to estimate each of the integrals
$\int_{F_\eps^i} |G_\eps|d\tau_\eps$
and
$\int_{H_\eps} |G_\eps|d\tau_\eps$.

We next estimate $\int_{F_\eps^n} |G_\eps|d\tau_\eps$.
Observe that
$F_\eps^n - \eps e_n\subset \{x\in\bR^n: x_n=0\}$.
Hence, $\vec{n}_\eps(x)=-e_n$.
In turn, this and \eqref{rep-G-eps} yield
\begin{align*}
G_\eps(x)
&= -\tilde{h}(x)\ww(x)X_n(x)
\\
&= 4\prod_{\ell=1}^{n-1}  x_\ell^{\kappa_\ell-1/2}(1-|x|)^{\kappa_{n+1}-1/2}
x_n^{\kappa_i+1/2}\big[\partial_n \tilde f(x)-\sum_{j=1}^nx_j\partial_j\tilde f(x)\big]
\end{align*}
and using the fact that $f$ is a polynomial and $h\in L^\infty$ we get
\begin{equation*}
|G_\eps(x)| \le c\eps^{\kappa_n+1/2}\prod_{\ell=1}^{n-1}x_\ell^{\kappa_\ell-1/2}(1-|x|)^{\kappa_{n+1}-1/2},
\quad x\in F_\eps^n.
\end{equation*}
Denote
\begin{equation*}
\tilde{F}_\eps^{n-1}:=\Big\{x\in \bR^{n-1}: x_1> \eps,\dots, x_{n-1}>\eps,\; \sum_{j=1}^{n-1}x_j< 1-2\eps \Big\},
\end{equation*}
which is the projection of $F_\eps^n$ onto $\bR^{n-1}=\{x\in\bR^n: x_n=0\}$.
With the notation $x':=(x_1, \dots, x_{n-1})$ and $|x'|:=x_1+\dots+x_{n-1}$, we have
\begin{align}\label{int-Wn}
\int_{F_\eps^n} |G_\eps|d\tau_\eps
&= \int_{\tilde{F}_\eps^{n-1}}|G_\eps(x_1, \dots, x_{n-1}, \eps)|dx'\notag
\\
&\le c\eps^{\kappa_n+1/2}\int_{\tilde{F}_\eps^{n-1}} \prod_{\ell=1}^{n-1}x_\ell^{\kappa_\ell-1/2}(1-\eps-|x'|)^{\kappa_{n+1}-1/2}dx'
\\
&\le c\eps^{\kappa_n+1/2}\int_{\bT^{n-1}_{1-\eps}}\prod_{\ell=1}^{n-1}x_\ell^{\kappa_\ell-1/2}(1-\eps-|x'|)^{\kappa_{n+1}-1/2} dx'
\le c'\eps^{\kappa_n+1/2}.\notag
\end{align}
Here for the former inequality we used that $\tilde{F}_\eps^{n-1} \subset \bT^{n-1}_{1-\eps}$
and for the latter we used \eqref{w-meas-T}.
We similarly obtain
\begin{equation}\label{est-G-Fi}
\int_{F_\eps^i} |G_\eps|d\tau_\eps \le c\eps^{\kappa_i+1/2}
\quad\hbox{for}\;\; i\ne n.
\end{equation}

We now estimate $\int_{H_\eps} |G_\eps|d\tau_\eps$.
Clearly, $\vec{n}_\eps(x)=\frac{1}{\sqrt{n}}(1, \dots, 1)$
is the unit outward normal vector to $\partial V_\eps$ at each $x\in H_\eps$.
This and \eqref{def-Xi}-\eqref{rep-G-eps} imply that for $x\in H_\eps$
\begin{align*}
G_\eps(x)&= \frac{1}{\sqrt{n}}\tilde{h}(x)\ww(x)\sum_{i=1}^n X_i(x)
\\
&=\frac{4}{\sqrt{n}}\tilde{h}(x)\prod_{\ell=1}^n  x_\ell^{\kappa_\ell-1/2}(1-|x|)^{\kappa_{n+1}-1/2}
\sum_{i=1}^n x_i\Big[\partial_i \tilde f(x)-\sum_{j=1}^nx_j\partial_j\tilde f(x)\Big]\\
&=\frac{4}{\sqrt{n}}\tilde{h}(x)\prod_{\ell=1}^n  x_\ell^{\kappa_\ell-1/2}(1-|x|)^{\kappa_{n+1}+1/2}
\sum_{j=1}^nx_j\partial_j\tilde f(x)
\end{align*}
and hence
$|G_\eps(x)| \le c\eps^{\kappa_{n+1}+1/2}\prod_{\ell=1}^n  x_\ell^{\kappa_\ell-1/2}$,
$x\in H_\eps$.
The surface $H_\eps$ can be described by the equation
\begin{equation*}
x_n=1-\eps- x'\;\;\hbox{for}\;\; x':=(x_1, \dots, x_{n-1})\in \hat{F}_\eps^{n-1},
\end{equation*}
where
$\hat{F}_\eps^{n-1}:=\big\{x\in \bR^{n-1}: x_1> \eps, \dots, x_{n-1}> \eps,\; \sum_{j=1}^{n-1}x_j< 1-\eps \big\}$.
Therefore,
\begin{align*}
\int_{H_\eps} |G_\eps|d\tau_\eps
&= \sqrt{n}\int_{\hat{F}_\eps^{n-1}}|G_\eps(x_1, \dots, x_{n-1}, 1-\eps-|x'|)|dx'
\\
&\le c\eps^{\kappa_{n+1}+1/2}\int_{\hat{F}_\eps^{n-1}}\prod_{\ell=1}^{n-1} x_\ell^{\kappa_\ell-1/2}(1-\eps-|x'|)^{\kappa_n-1/2}dx'
\\
& \le c\eps^{\kappa_{n+1}+1/2}\int_{\bT^{n-1}_{1-\eps}}\prod_{\ell=1}^{n-1} x_\ell^{\kappa_\ell-1/2}(1-\eps-|x'|)^{\kappa_n-1/2}dx',
\end{align*}
where we used that $\hat{F}_\eps^{n-1} \subset \bT^{n-1}_{1-\eps}$.
We use again \eqref{w-meas-T} to obtain
\begin{equation}\label{int-H}
\int_{H_\eps} |G_\eps|d\tau_\eps
\le c\eps^{\kappa_{n+1}+1/2}.
\end{equation}

Combining estimates \eqref{int-Wn}, \eqref{est-G-Fi}, and \eqref{int-H} we arrive at
\begin{equation*}
\int_{\partial V_\eps} |G_\eps|d\tau_\eps \le c\sum_{i=1}^{n+1}\eps^{\kappa_i+1/2}.
\end{equation*}
From this, taking into account that $\kappa_i>-1/2$, $i=1, \dots, n+1$, we conclude that
$\lim_{\eps\to 0}\int_{\partial V_\eps} |G_\eps|d\tau_\eps =0$.
The proof of Theorem~\ref{thm:divergence-T} is complete.
\end{proof}

\begin{remark}\label{rem:green-T}
Observe that Theorem~\ref{thm:divergence-T} and \cite[Proposition 3.1]{KPX} are equivalent.
Namely, it can be shown that identity $\eqref{divergence-T}$ can be derived from $(3.2)$ in \cite{KPX}
and vise versa.
\end{remark}

\noindent
{\em Completion of the proof of Theorem~\ref{thm:Gauss-simpl}.}
As was shown above the current setting on the simplex is covered by the general setting described in \S\ref{subsubsec:setting}
and above we verified conditions {\bf C0-C5}.
Therefore, Theorem~\ref{thm:Gauss-simpl} follows by Theorem~\ref{thm:main}.

\section{Jacobi heat kernel on $[-1, 1]$}\label{sec:jacobi}

The classical Jacobi operator is defined by
\begin{equation}\label{def-Jacobi}
Lf(x):=\frac{\big[w(x)(1-x^2)f'(x)\big]'}{w(x)}, %\quad D(L):=\tilde\cP[-1,1],
\end{equation}
where
\begin{equation*}
w(x):=(1-x)^{\alpha}(1+x)^{\beta}, \quad \alpha, \beta>-1.
\end{equation*}
We consider $L$ with domain $D(L):=\tilde\cP[-1,1]$ the set of all algebraic polynomials restricted to $[-1, 1]$.
We also consider $[-1, 1]$ equipped with the weighted measure
\begin{equation}\label{mes-int}
d\mu(x) := w(x) dx = (1-x)^{\alpha}(1+x)^{\beta} dx
\end{equation}
and the distance
\begin{equation}\label{dist-dist}
\rho(x,y) := |\arccos x - \arccos y|.
\end{equation}

We shall use the notation
\begin{equation}\label{int-ball}
B(x, r):=\{y\in [-1,1]: \rho(x, y)<r\} \quad\hbox{and}\quad V(x, r):=\mu(B(x, r)).
\end{equation}

As is well known \cite{Sz} the Jacobi polynomials $P_k$, $k\ge 0$, are eigenfunctions of $L$, that is,
\begin{equation}\label{Jacobi-eigenv}
LP_k= -\lambda_kP_k \quad\hbox{with}\quad \lambda_k=k(k+\alpha+\beta+1), \; k=0, 1, \dots.
\end{equation}
We consider the Jacobi polynomials $\{P_k\}$ normalised in $L^2([-1,1], \mu)$.
Then the Jacobi heat kernel $e^{tL}(x,y)$, $t>0$, takes the form
\begin{equation}\label{Jacobi-HK}
e^{tL}(x,y)=\sum_{k=0}^\infty e^{-\lambda_k t}P_k(x)P_k(y).
\end{equation}

%%%%%%%%% Theorem

\begin{theorem}\label{thm:Gauss-int}
The Jacobi operator $L$ in the setting described above
is essentially self-adjoint and $-L$ is positive.
Moreover, $e^{tL}$, $t>0$, is an integral operator whose kernel $e^{tL}(x, y)$ has Gaussian upper and lower bounds, that is,
there exist constants $c_1, c_2, c_3, c_4 >0$ such that for any $x,y \in [-1,1]$ and $t>0$
\begin{equation}\label{gauss-int}
\frac{c_1\exp\{- \frac{\rho(x,y)^2}{c_2t}\}}{\big[V(x, \sqrt t)V(y, \sqrt t)\big]^{1/2}}
\le  e^{tL}(x,y)
\le \frac{c_3\exp\{- \frac{\rho(x,y)^2}{c_4t}\}}{\big[V(x, \sqrt t)V(y, \sqrt t)\big]^{1/2}}.
\end{equation}
\end{theorem}

\begin{proof}
We shall derive estimate \eqref{gauss-int} from the two-sided estimate for the heat kernel on the simplex
(Theorem~\ref{thm:Gauss-simpl}) in dimension $n=1$ by changing the variables.
Assume $\alpha, \beta>-1$  and let $\beta=:\kappa_1-1/2$ and $\alpha=:\kappa_2-1/2$.
Clearly, $\kappa_1, \kappa_2>-1/2$.

We assume that $x_1\in [0,1]$. We shall apply the change of variables
\begin{equation}\label{change}
x_1=\frac{1}{2}(x+1),\; x\in [-1,1]
\quad\hbox{or}\quad
x=2x_1-1.
\end{equation}

The differential operator $L_T:=L$ from \eqref{def-L-simpl} in the case $n=1$ takes the form
\begin{equation*}
L_T=x_1\partial_1^2-x_1^2\partial_1^2 +(\kappa_1+1/2)\partial_1 - (\kappa_1+\kappa_2+1)x_1\partial_1
\end{equation*}
and hence for any $g\in C^2[0, 1]$
\begin{equation*}
L_Tg(x_1)= (x_1-x_1^2)g''(x_1) +(\beta+1)g'(x_1) - (\alpha+\beta+2)x_1g'(x_1).
\end{equation*}
Denote $f(x):= g((x+1)/2)$ or $g(x_1)=f(2x_1-1)$.
A little calculus shows that
\begin{equation}\label{L-L}
L_Tg(x_1)
%=L_Tg((x+1)/2)
= (1-x^2)f''(x) +(\beta-\alpha)f'(x) -(\alpha+\beta+2)xf'(x)
= Lf(x),
\end{equation}
where $L$ is the Jacobi operator from \eqref{def-Jacobi}.

Denote $d\mu_T(x_1):=x_1^{\kappa_1-1/2}(1-x_1)^{\kappa_2-1/2}d x_1$.
Let $\tilde{P}_k$, $k=0, 1, \dots$, be the orthogonal and normalized polynomials in $L^2([0,1],\mu_T)$.
From \eqref{sim-eigen-sp} we have
\begin{equation}\label{eigen-T}
L_T\tilde{P}_k = -\lambda_k \tilde{P}_k,
\quad\hbox{where}\quad
\lambda_k:=k(k+\kappa_1+\kappa_2)= k(k+\alpha+\beta+1).
\end{equation}
Now, by \eqref{L-L}, \eqref{Jacobi-eigenv}, and \eqref{eigen-T} we obtain
\begin{equation}\label{poly-poly}
P_k(x)= 2^{-(\alpha+\beta+1)/2}\tilde{P}_k((x+1)/2).
\end{equation}

Let $\rho_T(x_1, y_1) := \arccos\big(\sqrt{x_1y_1}+\sqrt{1-x_1}\sqrt{1-y_1}\big)$ be the distance on $[0,1]$
from \eqref{def-dist-simpl} when $n=1$.
We claim that
\begin{equation}\label{d-d}
\rho_T(x_1, y_1) = \rho(x, y)/2,
\quad\hbox{where}\quad x_1=(x+1)/2, \;\; y_1=(y+1)/2.
\end{equation}
Indeed, by applying cosine to both sides it is easy to see that
$$
|\arccos u- \arccos v|= \arccos\big(uv + \sqrt{(1-u^2)(1-v^2)}\big), \quad \forall u, v\in [-1, 1].
$$
Therefore,
\begin{align*}
\rho_T(x_1, y_1)
&= |\arccos \sqrt{x_1}- \arccos \sqrt{y_1}|
= \big|\arccos \sqrt{(x+1)/2}- \arccos \sqrt{(y+1)/2}\big|
\\
& = \Big|\int_{\sqrt{(x+1)/2}}^{\sqrt{(y+1)/2}}\frac{1}{\sqrt{1-s^2}} ds\Big|
= \frac{1}{2}\Big|\int_x^y\frac{1}{\sqrt{1-v^2}} dv\Big|
=\frac{1}{2}|\arccos x -\arccos y|, % =\frac{1}{2}\rho(x, y).
\end{align*}
which implies \eqref{d-d}.
For the former equality above we applied the substitution $s=\sqrt{(v+1)/2}$.

From \eqref{doubling-T2} it follows that for any $x_1\in [0, 1]$ and $0<r\le 1$ we have
\begin{align*}
\mu_T(B_T(x_1,r))
%&\sim r(\sqrt x_1+r)^{2\kappa_1} (r+\sqrt{1-x_1})^{2\kappa_2}
&\sim r(1-x_1+r^2)^{\kappa_2}(x_1+r^2)^{\kappa_1}
\\
&\sim r(1-x+r^2)^{\alpha+1/2} (1+x+r^2)^{\beta+1/2},
\quad x_1=(x+1)/2.
\end{align*}
On the other hand, it is easy to see that for any $x\in[-1, 1]$ (see \cite[(7.1)]{CKP})
$$
\mu(B(x,r))\sim r(1-x+r^2)^{\alpha+1/2} (1+x+r^2)^{\beta+1/2}.
$$
Combining the above we arrive at
\begin{equation}\label{meas-meas}
\mu_T(B(x_1,r)) \sim \mu(B(x,r)),
\quad\hbox{where}\quad x_1=(x+1)/2, \; x\in[-1, 1].
\end{equation}

We are now prepared to complete the proof of Theorem~\ref{thm:Gauss-int}.
From \eqref{poly-poly} it follows that
$$
e^{tL}(x, y) = 2^{-(\alpha+\beta+1)}e^{tL_T}(x_1, y_1),
\quad\hbox{where}\quad x_1=(x+1)/2, \;\; y_1=(y+1)/2.
$$
Therefore, using the two-sided Gaussian bounds on the heat kernel $e^{tL_T}(x_1, y_1)$ from Theorem~\ref{thm:Gauss-simpl},
\eqref{d-d}, and \eqref{meas-meas} we conclude that the Gaussian estimates \eqref{gauss-int} are valid.
\end{proof}

\begin{remark}
Theorem~\ref{thm:Gauss-int} is also proved in \cite[Theorem~7.2]{CKP} using a different but related approach.
A totaly different proof of Theorem~\ref{thm:Gauss-int} in the case $\alpha, \beta>-1/2$ is given in \cite{NS} using special functions.
It should also be pointed that in the case when $\alpha=\beta >-1$ estimates \eqref{gauss-int} follow readily
by the two-sided bounds for the heat kernel on the ball in dimension $n=1$ (Theorem~\ref{thm:Gauss-ball}).
\end{remark}


\begin{thebibliography}{99}

\bibitem{AUBIN}
    T. Aubin, Some nonlinear problems in Riemannian Geometry, Springer-Verlag, 1998.

\bibitem{CKP}
    T. Coulhon, G. Kerkyacharian, and P. Petrushev,
    Heat kernel generated frames in the setting of Dirichlet spaces,
    J. Fourier Anal. Appl. 18 (2012), 995--1066.

\bibitem{CS}
    T. Coulhon, A. Sikora,
    Gaussian heat kernel upper bounds via the Phragm\'en-Lindel\"of theorem,
    Proc. Lond. Math. Soc. 96 (2008), 507--544.

\bibitem{DaiXu}
    F. Dai, Y. Xu,
    Approximation theory and harmonic analysis on spheres and balls,
    Springer Monographs in Mathematics, Springer 2013.

\bibitem{DX}
    C. Dunkl, Y. Xu, Orthogonal polynomials of several variables,
    Encyclopedia of Mathematics and its Applications 81, Cambridge University Press, Cambridge, 2001.

\bibitem{FUKU}
    M. Fukushima, Y. Oshima, and M. Takeda,
    Dirichlet forms and symmetric Markov processes. Second revised and extended edition,
    de Gruyter Studies in Mathematics, 19, Walter de Gruyter and Co., Berlin, 2011.

\bibitem{F-J1}
    M. Frazier, B. Jawerth,
    Decomposition of Besov spaces, Indiana Univ. Math. J. 34 (1985), 777--799.

\bibitem{F-J2}
    M. Frazier, B. Jawerth,
    A discrete transform and decomposition of distribution spaces,
    J. Funct. Anal. 93 (1990), 34--170.

\bibitem{F-J-W}
    M. Frazier, B. Jawerth, and G. Weiss,
    Littlewood-Paley theory and the study of function spaces,
    CBMS No 79 (1991), AMS.

\bibitem{Grig1}
    A. Grigor'yan, The heat equation on noncompact Riemannian manifolds,
    (Russian) Mat. Sb. 182 (1991), no. 1, 55--87; translation in Math. USSR-Sb. 72 (1992), no. 1, 47--77.

\bibitem{Grig}
    A. Grigor'yan,
    Heat kernel and analysis on manifolds,
    AMS/IP Studies in Advanced Mathematics, 47. American Mathematical Society, Providence, RI;
    International Press, Boston, MA, 2009.

\bibitem{GS}
    P. Gyrya, L. Saloff-Coste,
    Neumann and Dirichlet heat kernels in inner uniform domains,
    Ast\'erisque 336 , Soci\'et\'e Math\'ematique de France, 2011.

\bibitem{IPX1}
    K. Ivanov, P. Petrushev, and Y. Xu,
    Sub-exponentially localized kernels and frames induced by orthogonal expansions,
    Math. Z. 264 (2010), 361--397.

\bibitem{GKPP}
    G. Kerkyacharian, P. Petrushev,
    Heat kernel based decomposition of spaces of distributions in the framework of Dirichlet spaces.
    Trans. Amer. Math. Soc. 367 (2015), 121--189.

\bibitem{KPX}
    G. Kerkyacharian, P. Petrushev, and Y. Xu,
    Gaussian bounds for the heat kernels on the ball and simplex: Classical approach, preprint.

\bibitem{KLING}
    W. Klingenberg,
    Riemannian geometry, Walter de Gruyter Studies in Mathematics 1,
    Berlin-New York 1982.

\bibitem{KyrP}
    G. Kyriazis, P. Petrushev,
    "Compactly'' supported frames for spaces of distributions on the ball,
    Monatsh. Math. 165 (2012), 365--391.

\bibitem{KPX1}
    G. Kyriazis, P. Petrushev, and Y. Xu,
    Jacobi decomposition of weighted Triebel-Lizorkin and Besov spaces,
    Studia Math. 186 (2008), 161--202.

\bibitem{KPX2}
    G. Kyriazis, P. Petrushev, and Y. Xu,
    Decomposition of weighted Triebel-Lizorkin and Besov spaces on the ball,
    Proc. London Math. Soc. 97 (2008), 477--513.

\bibitem{LANG}
    S. Lang, Fundamentals of differential geometry,
    Springer-Verlag, New York, 1999.

\bibitem{NS}
    A. Nowak, P. Sj\"{o}gren,
    Sharp estimates of the Jacobi heat kernel,
    Studia Math. 218 (2013), no. 3, 219--244.

\bibitem{PX1}
    P. Petrushev, Y. Xu,
    Localized polynomial frames on the interval with Jacobi weights,
    J. Four. Anal. Appl. 11 (2005), 557--575.

\bibitem{PX2}
    P. Petrushev, Y. Xu,
    Localized polynomial frames on the ball,
    Constr. Approx. 27 (2008), 121--148.

\bibitem{S-C1}
    L. Saloff-Coste, A note on Poincar\'{e}, Sobolev, and Harnack inequalities,
    Internat. Math. Res. Notices (1992), no. 2, 27--38.

\bibitem{S-C2}
    L. Saloff-Coste,
    The heat kernel and its estimates. Probabilistic approach to geometry, 405--436,
    Adv. Stud. Pure Math. 57, Math. Soc. Japan, Tokyo, 2010.

\bibitem{FS}
    F. Sauvigny,
    Partial differential equations 1: Foundations and integral representations,
    Universitext, Springer-Verlag, Berlin, 2006.

\bibitem{Sz}
    G. Szeg\"o,
    Orthogonal Polynomials,
    Amer. Math. Soc. Colloq. Publ. Vol. 23, Providence, 4th edition, 1975.

\end{thebibliography}
\end{document}